\documentclass[11pt,reqno]{amsart}

\usepackage{amsthm,amssymb}
\usepackage[pagebackref,colorlinks,linkcolor=red,citecolor=blue,urlcolor=blue,hypertexnames=false]{hyperref}
\usepackage{amsrefs}
\usepackage{amscd}

\input{xy}
\xyoption{all}
\newdir{ >}{{}*!/-5pt/@{>}}

\numberwithin{equation}{subsection}

\theoremstyle{plain}
\newtheorem{thm}[equation]{Theorem}
\newtheorem{prop}[equation]{Proposition}
\newtheorem{coro}[equation]{Corollary}
\newtheorem{lem}[equation]{Lemma}

\theoremstyle{definition}
\newtheorem{defi}[equation]{Definition}

\theoremstyle{remark}
\newtheorem{nota}[equation]{Notation}
\newtheorem{rem}[equation]{Remark}

\newtheorem*{ack}{Ackknowledgements}
\theoremstyle{definition}
\newtheorem{exa}[equation]{Example}
\newtheorem{exas}[equation]{Examples}

\DeclareMathOperator{\hofi}{hofi}
\newcommand{\triqui}{\vartriangleleft}
\newcommand{\troqui}{\vartriangleright}
\newcommand{\ass}{\C-\mathrm{Alg}}
\newcommand{\bass}{\mathrm{BAlg}}
\newcommand{\inc}{\mathrm{inc}}
\newcommand{\conj}{\mathrm{conj}}
\renewcommand{\inf}{\mathrm{inf}}
\newcommand{\ninf}{\mathrm{ninf}}
\renewcommand{\top}{\mathrm{top}}

\newcommand{\nil}{\mathrm{nil}}
\newcommand{\dif}{\mathrm{dif}}
\newcommand{\alg}{\mathrm{alg}}

\newcommand*{\colim}{\mathop{\mathrm{colim}}}

\newcommand{\iso}{\overset\cong\longrightarrow}
\newcommand{\ul}{\underline}

\newcommand{\bpf}{\begin{proof}}
\newcommand{\epf}{\end{proof}}
\newcommand{\bprop}{\begin{prop}}
\newcommand{\eprop}{\end{prop}}
\newcommand{\bthm}{\begin{thm}}
\newcommand{\ethm}{\end{thm}}
\newcommand{\brk}{\begin{rem}}
\newcommand{\erk}{\end{rem}}
\newcommand{\bdefi}{\begin{defi}}
\newcommand{\edefi}{\end{defi}}
\newcommand{\blemma}{\begin{lem}}
\newcommand{\elemma}{\end{lem}}
\newcommand{\bclly}{\begin{coro}}
\newcommand{\eclly}{\end{coro}}
\newcommand{\bnota}{\begin{nota}}
\newcommand{\enota}{\end{nota}}

\newcommand{\be}{\begin{enumerate}}
\newcommand{\ee}{\end{enumerate}}
\newcommand{\topdf}{\texorpdfstring}

\newcommand{\comment}[1]{}
\newcommand{\ab}{\mathfrak{Ab}}

\newcommand{\elli}{\ell^\infty}
\newcommand{\ellp}{\ell^p}
\newcommand{\ellq}{\ell^q}

\newcommand{\Ell}{\mathcal{L}}
\newcommand{\Ellp}{\mathcal{L}^p}

\newcommand{\cA}{\mathcal{A}}
\newcommand{\cB}{\mathcal{B}}

\newcommand{\cD}{\mathcal{D}}

\newcommand{\cJ}{\mathcal{J}}
\newcommand{\cK}{\mathcal{K}}
\newcommand{\cI}{\mathcal{I}}

\newcommand{\cM}{\mathcal{M}}

\newcommand{\ran}{\mbox{ran}}

\newcommand{\sran}{\mbox{\tiny{ran}}}
\newcommand{\dom}{\mbox{dom}}
\newcommand{\sdom}{\mbox{\tiny{dom}}}
\newcommand{\diag}{\mathrm{diag}}
\newcommand{\emb}{\mathrm{Emb}}

\newcommand{\aij}{A_{ij}}

\newcommand{\Gamc}{\Gamma(\CC)}

\newcommand{\Gami}{\Gamma^\infty}
\newcommand{\Gamic}{\Gamma^\infty(\CC)}

\newcommand{\im}{\mbox{Im}}

\newcommand{\mspan}{\mbox{span}}

\newcommand{\hotimes}{\hat{\otimes}}
\newcommand{\sotimes}{\overset{\sim}{\otimes}}

\newcommand{\V}{\mathbb{V}}
\newcommand{\CC}{\mathbb{C}}
\newcommand{\C}{\CC}
\newcommand{\Q}{\mathbb{Q}}

\newcommand{\NN}{\mathbb{N}}
\newcommand{\N}{\NN}

\newcommand{\Z}{\mathbb{Z}}

\newcommand{\supp}{\mbox{supp}}
\newcommand{\cl}{\mathcal{B}}

\newcommand{\cP}{\mathcal{P}}

\newcommand{\fH}{\Gamma}
\newcommand{\fA}{\mathfrak{A}}
\newcommand{\fB}{\mathfrak{B}}

\newcommand{\ev}{\mathrm{ev}}

\title{Homotopy invariance through small stabilizations}
\author{Beatriz Abadie}

\author{Guillermo Corti\~nas}
\address{Dep. Matem\'atica-IMAS, FCEyN-UBA\\ Ciudad Universitaria Pab 1\\
1428 Buenos Aires\\ Argentina}
\email{gcorti@dm.uba.ar}\urladdr{http://mate.dm.uba.ar/\~{}gcorti}

\address{Centro de Matem\'atica, Facultad de Ciencias, Universidad de la Rep\'ublica\\
Igu\'a $4225$, 11400 Montevideo, Uruguay}
\email{abadie@cmat.edu.uy}

\thanks{Both authors were supported by MathAmSud network U11MATH-05, partially funded by ANII, Uruguay, and by MINCyT, Argentina. Corti\~nas was partially supported by CONICET and
by grants UBACyT W386, PIP 112-200801-00900 and MTM2007-64704 (FEDER funds).}

\begin{document}

\begin{abstract}
We associate an algebra $\Gami(\fA)$ to each bornological algebra
$\fA$. The algebra $\Gami(\fA)$ contains a two-sided ideal
$I_{S(\fA)}$ for each symmetric ideal $S\triqui\elli$ of bounded
sequences of complex numbers. In the case of $\Gami=\Gami(\C)$,
these are all the two-sided ideals, and $I_S\mapsto J_S=\cB I_S\cB$
gives a bijection between the two-sided ideals of $\Gami$ and those of
$\cB=\cB(\ell^2)$. We prove that Weibel's $K$-theory groups $KH_*(I_{S(\fA)})$
are homotopy invariant for certain ideals $S$ including $c_0$ and $\ell^p$.
Moreover, if either $S=c_0$ and $\fA$ is a local $C^*$-algebra or
$S=\ell^p,\ell^{p\pm}$ and $\fA$ is a local Banach algebra, then
$KH_*(I_{S(\fA)})$ contains $K_*^{\top}(\fA)$ as a direct summand.
Furthermore, we prove that for $S\in\{c_0,\ell^p,\ell^{p\pm}\}$
there is a long exact sequence
\[
\xymatrix{
KH_{n+1}(I_{S(\fA)})\ar[r]&HC_{n-1}(\Gami(\fA):I_{S(\fA)})\ar[d]\\
KH_{n}(I_{S(\fA)})&K_n(\Gami(\fA):I_{S(\fA)})\ar[l]
}
\]

\end{abstract}

\maketitle

\section{Introduction}
\numberwithin{equation}{section}
Let $\ell^2=\ell^2(\N)$ be the Hilbert space of square-summable
sequences of complex numbers and $\cB=\cB(\ell^2)$ the algebra of
bounded operators. Let $\emb$ be the inverse monoid of all partially defined
injections
\[
\N\supset\dom f\overset{f}{\longrightarrow}\N.
\]
Each element $f\in\emb$
defines a partial isometry $U_f\in\cB$; for the canonical Hilbert
basis we have $U_f(e_n)=e_{f(n)}$ if $n\in\dom f$ and $0$ otherwise.
Similarly, each bounded sequence of complex numbers $\alpha\in\elli$
defines an element $\diag(\alpha)\in\cB$ by
$\diag(\alpha)(e_n)=\alpha_ne_n$. The subspace generated by all the
$U_f$ and $\diag(\alpha)$ with $f\in\emb$ and $\alpha\in\elli$ is
the subalgebra
\[
\cB\supset\Gami:=\mspan\{\diag(\alpha)U_f:\alpha\in\elli, f\in\emb\}.
\]
In this article we show that the algebra $\Gami$ has several remarkable properties.
One of them is that the lattice of
two-sided ideals of $\Gami$ is isomorphic to the lattice of two-sided ideals of $\cB$.
A theorem of Calkin (\cite{calk}), as restated by Garling (\cite{garling}), establishes a
one-to-one correspondence between two-sided ideals of $\cB$ and the
ideals of $\elli$ that are symmetric, that is, invariant under the
action of $\emb$. Calkin's correspondence
maps a symmetric ideal $S\triqui\elli$ to the ideal
$J_S$ of those operators whose sequence of singular values belongs
to $S$. Consider the subspace
 \[
\Gami\supset I_S:=\mspan\{\diag(\alpha)U_f:\alpha\in S, f\in\emb\}.
\]
Note that $I_{\elli}=\Gami$; for all symmetric ideals $S$, $I_S\triqui\Gami$ is a two-sided ideal. We prove (see Theorem \ref{calkinforgami})

\begin{thm}\label{intro:calkinforgami}
The map $J\mapsto J\cap\Gami$ is a bijection between the sets of
two-sided ideals of $\cB(\ell^2(\NN))$ and $\Gami$. If
$S\triqui\elli$ is a symmetric ideal, then $J_S\cap \Gami=I_S$.
\end{thm}

More generally, we define for any bornological algebra $\fA$ (in
particular for a Banach algebra $\fA$) an algebra $\Gami(\fA)$. The
algebra $\Gami (\fA)$ contains an ideal $I_{S(\fA)}$ for any
symmetric ideal $S\triqui\elli$, and $S\mapsto I_{S(\fA)}$ is a
lattice homomorphism. Thus the smallest nonzero 
$I_{S(\fA)}$ occurs when $S$ is the symmetric ideal
$c_f\triqui\elli$ of finitely supported sequences; we get
\[
I_{c_f(\fA)}=M_\infty \fA=\bigcup_nM_n\fA.
\]
Hence the inclusion $\fA\to M_\infty \fA$ into the upper left corner
gives a stability homomorphism
\begin{gather*}
\iota_S:\fA\to I_{c_f(\fA)}\subset I_{S(\fA)}.
\end{gather*}
If $\fA$ is unital then $\iota_{c_f}$ induces an isomorphism in
algebraic $K$-theory, by matrix stability. At the other extreme, 
$I_{\elli(\fA)}=\Gami(\fA)$ is a ring with infinite sums in the
sense of \cite{wag} (see Proposition \ref{prop:gaminfisuri}); this permits the Eilenberg swindle and we have
\[
K_*(\Gami(\fA))=0.
\]
For $c_f\subsetneq S\subsetneq \elli$, the $K$-theory of
$I_{S(\fA)}$ is more interesting. We study it for
\begin{equation}\label{eq:lists}
S\in\{c_0, \ell^{p-},\ellq, \ell^{q+} \quad(p\le\infty, q<\infty)\}.
\end{equation}
Here $c_0$ is the ideal of sequences vanishing at infinity,
$\ellq$ consists of the $q$-summable sequences, and
\[
\ell^{p-}=\bigcup_{r<p}\ell^r,\ \ \ell^{q+}=\bigcap_{s>q}\ell^s.
\]
 Let $\bass$ be the category of bornological
algebras. We consider several variants of $K$-theory. We write $K$
for algebraic $K$-theory, $KH$ for Weibel's homotopy algebraic
$K$-theory and $K^{\top}$ for topological $K$-theory. The following
result follows from Theorem \ref{thm:kh0}.

\begin{thm}\label{thm:kh0intro}
\item[i)] The functor $\bass\to \ab$, $\fA\mapsto KH_*(I_{c_0(\fA)})$
is invariant under continuous homotopy.

\item[ii)] If $\fA$ is a local $C^*$-algebra and $n\ge 0$, then
there is a natural split monomorphism
\[
\xymatrix{K^{\top}_n(\fA)\ar@{ >-}[r]& KH_n(I_{c_0(\fA)})}.
\]
\item[iii)] If $n\le 0$, then the comparison map
\begin{equation}\label{intro:compa}
K_n(I_{c_0(\fA)})\to KH_n(I_{c_0(\fA)})
\end{equation}
is an isomorphism for every $\fA\in\bass$. 
\end{thm}

The results above should be compared with Karoubi's conjecture
(Suslin-Wodzicki's theorem \cite{sw1}*{Theorem 10.9}) that for a $C^*$-algebra
$\fA$, the comparison map
\[
K_*(\fA\sotimes\cK)\to K^{\top}_*(\fA\sotimes\cK)\cong
K^{\top}_*(\fA)
\]
is an isomorphism. Hence we may think of $\fA\to I_{c_0(\fA)}$ as a
smaller version of the stabilization $\fA\mapsto\fA\sotimes\cK$
whose homotopy algebraic $K$-theory is continuously homotopy invariant and contains $K_*^{\top}(\fA)$
as a direct summand. Next let
$p\ge 1$ and consider the Schatten ideal $\Ellp\triqui\cB$. Notice that $\Ellp$
is the ideal corresponding to $\ellp$ under Calkin's correspondence. We have
\[
\Ellp=J_{\ellp}.
\]
Recall from \cite{cot}*{Theorem 6.2.1} that if $\fA$ is a locally convex algebra and
$\fA\hotimes\Ellp$ is the projective tensor product then
\[
KH_*(\fA\hotimes\Ell^1)\iso KH_*(\fA\hotimes\Ellp)\iso
K^{\top}_*(\fA\hotimes\Ellp).
\]
In the present article (Theorem \ref{thm:kh}) 
we prove the following analogue of the latter
result.

\begin{thm}\label{thm:khintro}
Let $S$ be one of $\ell^p$, $\ell^{p+}$ ($0<p<\infty$) or
$\ell^{p-}$ ($0<p\le\infty$).
\item[i)] The functor $\bass\to \ab$, $\fA\mapsto KH_*(I_{\ell^1(\fA)})$
is invariant under H\"older-continuous homotopies and we have
$KH_*(I_{S(\fA)})=KH_*(I_{\ell^1(\fA)})$ for all $S$ as above.
\item[ii)] If $\fA$ is a local Banach algebra and $n\ge 0$, then
there is a natural split monomorphism
\[
\xymatrix{K^{\top}_n(\fA)\ar@{ >-}[r]& KH_n(I_{\ell^1(\fA)})}.
\]
\item[iii)] If $n\le 0$, then the comparison map
\begin{equation}\label{intro:compa2}
K_n(I_{S(\fA)})\to KH_n(I_{S(\fA)})
\end{equation}
is an isomorphism for every $\fA\in\bass$. 
\end{thm}

Both these theorems rely on a homotopy invariance theorem (Theorem \ref{thm:htpy})
which we think is of independent interest. The theorem says that if $F:\ass\to\ab$ is
an $M_2$-stable, split exact functor and $S\in\{c_0,\ellp\}$,  then the functor
\[
\bass\to\ab,\ \ \fA\mapsto F(I_{S(\fA)})
\]
is homotopy invariant. For $S=c_0$ it is continuous homotopy
invariant, while for $S=\ellp$ it is invariant under H\"older
continuous homotopies, with H\"older exponent depending on $p$. For
$F=KH_*$ we have $KH_*(I_{\ellp(\fA)})=KH_*(I_{\ell^1(\fA)})$, and so
it is invariant under arbitrary H\"older continuous homotopies.
Furthermore, we have the following general result (see Theorem
\ref{thm:kseq}) about the comparison map $K\to KH$. Its proof uses
the homotopy invariance theorem mentioned above applied to
infinitesimal $K$-theory.

\begin{thm}\label{thm:introkseq}
Let $\fA$ be a bornological algebra and let $S$ be $c_0$,
$\ell^p$, $\ell^{p+}$ ($0<p<\infty$) or $\ell^{p-}$
($0<p\le\infty$). Then there are long exact sequences ($n\in\Z$)
\begin{equation}\label{introseq:abs}
\xymatrix{
KH_{n+1}(I_{S(\fA)})\ar[r]&HC_{n-1}(I_{S(\fA)})\ar[d]\\
KH_{n}(I_{S(\fA)})&K_n(I_{S(\fA)})\ar[l]}
\end{equation}
and
\begin{equation}\label{introseq:rel}
\xymatrix{
KH_{n+1}(I_{S(\fA)})\ar[r]&HC_{n-1}(\Gami(\fA):I_{S(\fA)})\ar[d]\\
KH_{n}(I_{S(\fA)})&K_n(\Gami(\fA):I_{S(\fA)})\ar[l] }
\end{equation}
\end{thm}

It is shown in the companion paper \cite{whatever} that $HC_*(\Gami(\fA):I_{S(\fA)})=0$ when either $S=c_0$ and $\fA$ is a C*-algebra or $S=\elli$ and $\fA$ is a unital Banach algebra. Therefore, the comparison  map $K_*(I_{S(\fA)})\longrightarrow KH_*(I_{S(\fA)})$ is an isomorphism in these cases. In addition, the groups $HC_n(\Gami:I_S)$ are computed in \cite{whatever} for $S\in\{\ell^p,\ell^{p\pm}\}$, and the map $HC_n(\Gami:I_S)\longrightarrow HC_n(\cB:J_S)$ is shown to be an isomorphism for those values of $n$ for which $HC_n(\cB:J_S)$ was computed by Wodzicki
(\cite{wodk}).

The rest of this paper is organized as follows. In Section
\ref{sec:prelis} we establish some notation about sequence spaces,
the inverse monoid $\emb$ and the partial isometries $U_f$. The
algebra $\Gami(\fA)$ and the ideals $I_{S(\fA)}$ are introduced in
Section \ref{sec:gamis}. In this section we also recall the
definition of Karoubi's cone $\Gamma(R)$ which is $R$-linearly
generated by the $U_f$ ($f\in\emb$). Proposition \ref{gamma}
identifies $I_{S(\fA)}$ with a ring formed by certain $\N\times\N$
matrices with coefficients in $\fA$. The two-sided
ideals of $\Gami$ are studied in Section \ref{sec:ideals}; Theorem
\ref{intro:calkinforgami} is contained in Theorem
\ref{calkinforgami}. We prove in Section \ref{sec:misce} that if $\fA$ is unital, then $\Gami(\fA)$ is a ring
with infinite sums in the sense of Wagoner (Proposition
\ref{prop:gaminfisuri}). In Section \ref{sec:crossp} we show that
$I_{S(\fA)}$ can be written as a crossed product of
$\Gamma=\Gamma(\Z)$ and $S(\fA)$, by using the conjugation action of
$\emb$ in $S(\fA)$ via the partial isometries $U_f$ (Proposition
\ref{prop:cpg}). Section \ref{sec:hit} deals with the
homotopy invariance theorem mentioned above, proved in Theorem
\ref{thm:htpy}. Applications to $K$-theory are given in Section
\ref{sec:K}; see Theorems \ref{thm:kh}, \ref{thm:kh0} and
\ref{thm:kseq}. 
\begin{ack}
Most of the research for this paper was carried out during visits of
B. Abadie to the Universidad de Buenos Aires and of G. Corti\~nas to
the Universidad de la Rep\'ublica. We are thankful to these
institutions for their hospitality. G. Corti\~nas wishes to thank
his colleague Daniel Carando for useful
discussions about topological tensor products and Ruy Exel for many
useful discussions and for patiently explaining his paper \cite{ruy}.
\end{ack}

\section{Preliminaries}\label{sec:prelis}

\subsection{Sequence ideals}

\numberwithin{equation}{subsection}

Throughout this paper we work in the setting of bornological spaces
and bornological algebras; a quick introduction to the subject is
given in \cite{cmr}*{Chapter 2}. Recall a (complete, convex)
bornological vector space over the field $\C$ of complex numbers is
a filtering union $\V=\cup_D\V_D$ of Banach spaces, indexed by the
disks of $\V$ such that the inclusions $\V_D\subset \V_{D'}$ are
bounded.
A subset of $\V$ is \emph{bounded} if it is a
bounded subset of some $\V_D$. A sequence $\N\to \V$ is
\emph{bounded} if its image is a bounded subset of $\V$. We write
$\elli(\N,\V)$ or simply $\elli(\V)$ for the bornological vector
space of bounded sequences where $X\subset\elli(\V)$ is bounded if
$\bigcup_{x\in X}x(\N)$ is. We consider the following closed
bornological subspace
\begin{equation}
\label{c0}
\elli(\V)\supset c_0(\V)=\{\alpha:\lim_n\alpha_n=0\}
\end{equation}
We also consider the subspace ($p>0$)
\[
c_0(\V)\supset \ellp(\V)=\{\alpha:\N\to \V:(\exists \mbox{ a disk } D\subset \V)\sum_n||\alpha_n||_D^p<\infty\}
\]
If $p\ge 1$, we equip $\ellp(\V)$ with the following bornology: we
say that a subset $S\subset \ellp(\V)$ is bounded if there exist a
disk $D$ and a constant $C$ such that $\sum_n||\alpha_n||_D^p<C$ for
all $\alpha\in S$. Notice that the inclusion $\ellp(\V)\to
\elli(\V)$ is bounded for $p\ge 1$. Recall a bornological algebra is
a bornological vector space $\fA$ with an associative bounded
multiplication. If $\fA$ is a bornological algebra, then pointwise
multiplication makes $\elli(\fA)$ into a bornological algebra,
$c_0(\fA)\triqui\elli(\fA)$ is a closed bornological ideal, and
$\ellp(\fA)\triqui\elli(\fA)$ is an algebraic ideal for all $p>0$. '

\begin{nota}\label{nota:noC}
When $\fA$ is $\C$, we shall omit it from our notation. Thus we shall write $\elli$, $\ell^p$, $c_0$, etc, for $\elli(\C)$, $\ell^p(\C)$, $c_0(\C)$, etc.
\end{nota}

The space $\cB(\ell^2(\V))$ of bounded operators $\ell^2(\V)\to \ell^2(\V)$ on a bornological vector space $\V$ is a bornological algebra with the uniform bornology (\cite{cmr}*{Def. 2.4}). If $\fA$ is a bornological algebra, then
\begin{equation}\label{defdiag}
\diag:\elli(\fA)\to \cB(\ell^2(\fA)),\quad \diag(\alpha)(\xi)=(\alpha_n\xi_n)_{n\ge 1}.
\end{equation}
is a bounded representation. It is faithful if and only if the
left annihilator of $\fA$ is trivial:
\[
\mathrm{ann}(\fA)=\{a\in\fA: a\cdot b=0\quad (\forall b\in\fA)\}=0,
\]
This happens, for instance, when $\fA$ is unital.

\subsection {The monoid \texorpdfstring{$\emb$}{emb}.}\label{subsec:emb}

We begin by recalling some definitions from \cite{kenetal2}. We
denote by $\emb$ the set of injective functions
\[\text{$\emb$}=\{f:A \rightarrowtail \NN: A\subset \NN\}.\]
Note that $\emb$ is a monoid for the composition law:
\begin{equation}
\label{prodemb} f g:\dom(g)\cap g^{-1}(\dom
(f))\rightarrow \NN,\ (f g)(n)=f(g(n)).
\end{equation}
In \eqref{prodemb} and elsewhere, we shall omit the composition sign
$\circ$, except when strictly necessary to avoid
confusion.
The monoid $\emb$ is
\emph{pointed}, i.e. it has a zero element, namely, the empty
function $\emptyset\to \N$. The antipode map $^\dagger:
\text{$\emb$} \rightarrow \text{$\emb$}$ is defined by
\[\dom(f^\dagger)=\ran(f),\  f^\dagger(n)=f^{-1}(n).\]
If $A\subset \N$, we write $P_A$ for the inclusion $A\hookrightarrow \N$.
It is easily checked that
\begin{equation}
\label{dagproj}
f^\dagger f =P_{\sdom f},\ f f^\dagger=P_{\sran f},
\end{equation}
for any $f\in$ $\emb$. Observe that $f^\dagger$ is characterized as
the unique element of $\emb$ which satisfies simultaneously
\begin{gather*}
f f^\dagger f=f\mbox{ and }
f^\dagger f
f^\dagger=f^\dagger.
\end{gather*}
Thus the monoid $\emb$ together with its antipode is a pointed
\emph{inverse monoid} that is, a pointed \emph{inverse semigroup}
with identity element. Note that $\emb$ is the object usually
denoted $\cI(\N)$ in the literature on semigroups (see
\cite{ruy}*{Def. 4.2}, for instance).

If $\V$ is a bornological vector space, the monoid $\emb$ acts on
$\elli(\V)$  via:
\begin{equation}
\label{actionofemb}
f_*(\alpha)_n= \begin{cases}
                    \alpha_{f^{\dagger}(n)}&\text{ if $n\in \ran (f)$ }\\
                         0& \text{otherwise}   .      \end{cases}
\end{equation}

The subspaces $c_0(\V)$ and $\ell^p(\V)$ defined in \ref{c0} are
\emph{symmetric}, i.e. they are invariant under the action of
$\emb$. Indeed, this follows from the fact that $c_0$ and $\ellp$
are symmetric, and that if $D$ is a bounded disk and the image of
$\alpha$ is contained in $\V_D$, then the following sequences of real numbers are identical
\[
||f_*(\alpha)||_D=f_*(||\alpha||_D).
\]
More generally, if $S\subset \elli$ is any symmetric
subspace, then
\[
S(\V):=\{\alpha\in\elli(\V):(\exists D)\, \alpha(\N)\subset \V_D\text{
and }||\alpha||_D\in S\}
\]
is symmetric. We
denote by $U$ the representation of $\emb$ by partial isometries on
$\ell^2(\V)$:
\begin{equation}
\label{defu}
U_f(\xi)_m=\begin{cases} \xi_n & \text{if $f(n)=m$}\\
                             0& \text{if $m\notin\ran(f)$}\end{cases}\qquad (\xi\in\ell^2(\V)).
\end{equation}

Straightforward computations show that
\begin{equation}\label{eqpropus}
U_{fg}=U_fU_g. 
\end{equation}

Observe that $U_f$ is a partial isometry whose initial and final space are, respectively, the closed subspaces
\[
\mspan\{v:\supp(v)\subset\dom(f)\}\mbox{ and } \mspan\{v:\supp(v)\subset \ran (f)\}.
\]
This follows from \eqref{dagproj}, \eqref{eqpropus}, and from the fact that if $A\subset \N$, then
\[
U_{P_{A}}(v)_n=\begin{cases}v_n&\text{if $n\in A$}\\
                              0 &\text{otherwise.}\end{cases}
\]

\brk\label{rk:s(x)} We will often work with sequences indexed by
infinite countable sets other than $\N$. A bijection $u:\N\to X$ gives rise
to a bounded isomorphism $\alpha\mapsto\alpha u$ between
the bornological vector space $\elli(X,\V)$ of bounded maps from $X$ to the bornological
space $\V$ and the space $\ell^\infty(\V)=\ell^\infty(\N,\V)$. If
$S\subset \ell^\infty$ is a symmetric subspace, we define
$S(X,\V)=\{s u^{-1}:s\in S(\V)\}$. Because $S$ is symmetric by assumption, this definition does
not depend on the choice of $u$. \erk

\begin{nota}\label{nota:noX}
Let $S\subset\elli$ be a symmetric subspace, $X$ an infinite countable set and $\V$ a bornological vector space. We use the following abbreviated notation: $S=S(\N,\C)$, $S(X)=S(X,\C)$ and $S(\V)=S(\N,\V)$.
\end{nota}
\section{The algebras \texorpdfstring{$\Gami(\fA)$}{Gami(fA)} and  \texorpdfstring{$\Gamma(R)$}{Gamma(R)}}\label{sec:gamis}
\setcounter{equation}{0} \numberwithin{equation}{section}

Throughout this section, $\fA$ will be a fixed bornological algebra,
which, except in Definition \ref{defi:gaminuni}, will be assumed
unital. It follows straightforwardly from equations \eqref{defdiag},
\eqref{actionofemb}, and \eqref{defu} that
\begin{equation}
\label{eqcov}
\begin{array}{lcr}
 \diag(f_*(\alpha))U_f=U_f\diag(\alpha) & \text{ and }&
U_f\diag(\alpha)U_{f^\dagger}= \diag(f_*(\alpha)),
\end{array}
\end{equation}
where $\alpha\in\elli(\fA)$ and $f\in\emb$. Set
\begin{equation}
\label{defgamma} \Gami(\fA)=\mspan\{\diag(\alpha)U_f:
\alpha\in\elli(\fA),\ \ f\in \text{$\emb$}\}.
\end{equation}
Notice that, by equations (\ref{eqpropus}) and (\ref{eqcov}),
$\Gami(\fA)$ is a subalgebra of the algebra $\cl(\ell^2(\fA))$.
 For each symmetric ideal $S\triqui\elli$, we write $I_{S(\fA)}$ for the ideal of $\Gami(\fA)$ generated by  $\diag(S(\fA))$. Because $S$ is invariant under the action of $\emb$, then by equations (\ref{eqcov}) we have
\begin{equation}\label{def:is}
I_{S(\fA)}=\mspan\{\diag(\alpha)U_f: \alpha\in S(\fA), f\in \text{$\emb$}\}.
\end{equation}
Note that $\Gami(\fA)=I_{\ell^\infty(\fA)}$. If $X$ is any infinite
countable set, we may also consider the subalgebra
$\Gami(X,\fA)\subset \cB(\ell^2(X,\fA))$ spanned by
$\diag(\elli(X,\fA))$ and $U_{\emb(X)}$. Thus
$\Gami(\fA)=\Gami(\N,\fA)$. In keeping with our notational
conventions \ref{nota:noC} and \ref{nota:noX}, we write
$\Gami=\Gamic$ and $\Gami(X)=\Gami(X,\C)$.

\begin{nota}
Since $\fA$ is assumed to be unital, every sequence $a=\{a_n\}$ in
$\ell^2(\fA)$ can be written uniquely as  $a=\sum_n a_ne_n$,  where
$e_n\in\ell^2(\fA)$ is defined by $(e_n)_i=\delta_{n,i}$. Notice
that the elements of $\Gami(\fA)$ are  $\fA$-linear operators on the
right $\fA$-module $\ell^2(\fA)$. As usual, we identify an
$\fA$-linear operator $A\in \cl(\ell^2(\fA))$ with the infinite
matrix $(A_{ij})_{i,j\in\N}$ with entries in $\fA$ defined by
\[Ae_n=\sum_kA_{kn}e_k.\]
We denote by $E_{ij}$ the matrix
$(E_{ij})_{kl}=\delta_{ik}\delta_{jl}$. Given a matrix
$A=(A_{ij})_{i,j\in\N}$ with entries in $\fA$, and $i,j\in \NN$, we
set:
\begin{gather*}
J_i(A)=\{j: A_{ij}\neq 0\}, I_j(A)=\{i: A_{ij}\neq 0\},\\
 r_i(A)=\# J_i(A), c_j(A):=\#I_j(A),\\
r(A):=\max _i r_i(A),\quad c(A):=\max _i c_i(A),\\
N(A):=\max\{r(A), c(A)\},
\end{gather*}
where $r_i(A),c_j(A),N(A)\in \NN\cup \{\infty\}$. If $R$ is a ring, we write
$\Gamma(R)$  for \emph{Karoubi's cone}
\begin{equation}\label{def:gama}
\Gamma(R)=\{A\in R^{\N\times\N}: N(A)<\infty\mbox{ and }
\{A_{i,j}:i,j\in\N\}\mbox{ is finite }\}.
\end{equation}
It was shown in (\cite{biva}*{Lemma 4.7.1}) that $\Gamma(R)$ is
isomorphic to $R\otimes\Gamma(\Z)$, for any ring $R$. We
shall write
\[
\Gamma=\Gamma(\Z).
\]
Observe that definition \eqref{def:gama} extends to matrices indexed
by any countable infinite set $X$; if $f:\N\to X$ is a bijection,
$\Gamma(X,R)\subset R^{X\times X}$ is the image of $\Gamma(R)$ under
the map $A\mapsto U_f AU_{f^{-1}}$. Thus $\Gamma(R)=\Gamma(\N,R)$;
we shall write $\Gamma(X)=\Gamma(X,\Z)$.
\end{nota}

The following lemmas will be useful in obtaining characterizations
of $\Gami(\fA)$, $I_{S(\fA)}$ and $\Gamma(R)$ as rings of matrices
acting on $\ell^2(\fA)$ and  $R^{(\N)}$, respectively. If $A\in
R^{\N\times\N}$ is such that $N(A)<\infty$, we write
$\Gamma(R)A\Gamma(R)$ to denote the set
\[
\Gamma(R)A\Gamma(R):=\{\sum_{j=1}^nP_jAQ_j:P_j, Q_j\in \Gamma(R) \mbox{ for all }j=1,\dots,n\mbox{ and } n\in\NN\}.
\]

\begin{lem}\label{elusive}

Let $R$ be a unital ring, $A=(A_{ij})_{i,j\in\NN}\in R^{\N\times\N}$ a matrix such that $N(A)<\infty$ and $r(A)>1$. Then
\be
\item $A=A_1+A_2+\cdots +A_k,$ where   $A_i\in \Gamma(R)A\Gamma(R)$, $r(A_i)< r(A)$
and $c(A_i)\leq c(A)$ for all $i=1,\dots,k$.
\item  If in addition $R$ is a unital bornological algebra and $S\triqui\elli$ is a symmetric ideal  such
that $\{A_{ij}\}\in S(\N\times\N,R)$, then $\{(A_l)_{ij}\}\in S(\N\times\N,R)$, for all $l=1,\dots, k$.
\ee
\end{lem}
\bpf
(1) We first establish some notation and make some reductions. Let
\begin{gather*}
 r=r(A)\\
I=\{i\in \NN: \text{ the $i^{th}$ row of $A$ has $r$ nonzero
entries}\}.
\end{gather*}
For $i\in I$, let
\[
h_i(1)<h_i(2)<\dots<h_i(r)
\]
be the columns where the nonzero entries of row $i$ occur. Let $A_r$
denote the matrix obtained from $A$ upon multiplying by zero those
rows that have less than $r$ nonzero entries. Then
$A_r\in\Gamma(R)A\Gamma(R)$, and
\[ r(A_r)= r,\ r(A-A_r)<r, \  c(A_r)\leq c(A),\mbox{ and }c(A-A_r)\leq c(A)   .\]
Thus it suffices to prove (1) for $A_r$. Hence we may assume that $A=A_r$, that is, that all nonzero rows of $A$ have exactly $r$ nonzero entries. Furthermore, since there are at most $c(A)$ nonzero entries in each
column of $A$,  the set $I$ can be written as a disjoint union $I=I_1\sqcup
I_2\sqcup\cdots\sqcup I_s$ with $s\leq c(A)$ and such that each $I_t$ ($1\le t\le s$) satisfies the following property:
$$i\neq j\in I_t\Rightarrow h_i(1)\neq h_j(1).$$
Proceeding as above we see that we may assume that $s=1$. Notice
that if $A'$ is obtained from $A$ by permuting its rows, then
$A'=U_fA$ for some bijection $f:\NN\rightarrow\NN$. Therefore, $
\Gamma(R)A\Gamma(R)=\Gamma(R)A'\Gamma(R)$, $r(A')=r(A)$, and
$c(A')=c(A)$, so we may assume that $A=A'$.  Thus we will assume
that the rows of $A$ are ordered so that if $i,j\in I$, then
$h_i(1)<h_j(1)$ if and only if $i<j$.

Thus, it only remains to show (1) for matrices $A$ such that
for $I$ and $h_i$ as above:
\begin{align}
\text{a) } &\mbox{All nonzero rows of $A$ have exactly $r$ nonzero entries.}\label{nonzero}\\
\text{b) } &i<j\iff h_i(1)<h_j(1)\mbox{ for all }i,j\in
I.\label{hinc}
\end{align}
We shall proceed by induction on
\[M_A= \max_{j\in I} \# \{i\in I: A_{ih_j(1)}\neq 0\}.\]
Notice that the right-hand side of the equation above is bounded by
$c(A)$, so $M_A\in\NN$. First assume that $M_A=1$. Then for all $i,j\in I$ we have that  $A_{ih_j(1)}\neq 0$ if and only if $i=j$. Set
\[A_1=\sum_{i\in I}A_{ih_i(1)}E_{ih_i(1)}=\big(\sum_{i\in I}E_{ii}\big)A\big(\sum_{j\in I}E_{h_j(1)h_j(1)}\big)\in\Gamma(R)A\Gamma(R).\]
Then
\[
r(A_1)< r,\   r(A-A_1) < r, \    c(A_1)\leq c(A),\mbox{ and } c(A-A_1)\leq c(A),
\]
so the statement in (1) holds for $A$. Assume now that $M_A>1$ and
that (1) holds for matrices $B$ satisfying \ref{nonzero} and
\ref{hinc}, and such that $M_B<M_A$. Let
\[
i_1:=\min I,\   K_1:=\{j\in I: A_{i_1h_j(1)\neq 0}\}.
\]
For $ n\geq 1$  such that $ \bigcup_{j=1}^{n-1}K_j\neq I$, let
\[
i_n:=\min I\setminus\bigcup_{j=1}^{n-1}K_j, \mbox{ and }
K_n:=\{j\in  I\setminus\bigcup_{l=1}^{n-1}K_l: A_{i_nh_j(1)} \neq 0\}.
\]
Let
\[\cJ=\begin{cases} \{1,2,\dots, n\},\text{ if }   \bigcup_{j=1}^{n}K_j=I.\\
                     \NN,\text{ otherwise.}\end{cases}    \]

We claim that
\begin{equation}\label{inc}
\text{a) }i_{n}>i_{n-1}\ \forall n\in \cJ \quad \text{ and
}\quad\text{  b) }\quad I=\bigcup_{j\in \cJ}K_j.
\end{equation}
In fact a) follows from the inequality
\[i_{n}=\min I\setminus \bigcup_1^{n-1}K_j \geq \min I\setminus  \bigcup_1^{n-2}K_j=i_{n-1}\]
and the fact that $i_n\neq i_{n-1}$ because $i_{n}\not\in K_{n-1}$
and $i_{n-1}\in K_{n-1}$. It is clear that b) holds when $\cJ$ is
finite. Assume now that $\cJ$  infinite. If $k\in I$, then either
$k\in \{i_n:n\in\cJ\}\subset \bigcup K_j$ or, by a), there exists
$n\in\cJ$ such that
\[k< i_n=\min I\setminus \bigcup_1^{n-1}K_j.\]
This implies that $k\in \bigcup_1^{n-1}K_j$. Thus b) holds also when
$\cJ$ is infinite, and both claims are proven. Now set
\[B:=\sum_{n\in \cJ,j\in \NN}A_{i_nj}E_{i_nj}=\big(\sum_{n\in \cJ}E_{i_ni_n}\big)A\in\Gamma(R)A\Gamma(R).\]
Notice that $B$ is obtained from $A$ by multiplying by zero the
$i^{th}$ row whenever $i\not\in\{i_n:n\in\cJ\}$. Therefore $B$
satisfies \ref{nonzero} and \ref{hinc}, $r(B)=r$, and $c(B)\leq
c(A)$. We next show that $M_B=1$. We begin by noting that  $B_{i_mi_n(1)}\neq 0$ implies that
$A_{i_mi_n(1)}\neq 0$. Then $i_n(1)\geq i_m(1)$, which implies by
\ref{hinc} that $i_n\geq i_m$, which in turn implies, by part a) of
equation (\ref{inc}), that $n\geq m$. Now, if $n>m$ we would have
\[i_n\not\in  \bigcup_1^{n-1}K_j \supseteq\bigcup_1^{m}K_j. \]
Then $i_n\not \in K_m$ and $i_n\not \in \bigcup_1^{m-1}K_j$, which
implies that $A _{i_mi_n(1)}=0$, a contradiction. Thus $n=m$ and $M_B=1$, as claimed.
Set $C=A-B$; we have $r(C)=r$ and  $c(C)\leq c(A)$.
 Notice that $C$ is obtained from $A$ upon multiplying by zero the $i_n^{th}$ row for all $n\in\cJ$. Besides,
 the $i^{th}$ row of $C$  is nonzero if and  only if $i\in I_C:=I\setminus\{i_n:n\in \cJ\}$, and in that case it is equal to the $i^{th}$ row of $A$. Therefore, $C$ satisfies \ref{nonzero} and \ref{hinc}.
We next prove that $M_C<M_A$, which will conclude the proof of part
(1).  If $i, j\in I_C$, then $A_{ih_j(1)} = 0$ implies that
$C_{ih_j(1)} = 0$. On the other hand,  by part b) of equation
(\ref{inc}), we can choose $n\in \cJ$  such that  $j\in K_n$. Then
$A_{i_nh_j(1)}\neq0$,  whereas $C_{i_nh_j(1)}=0$. It follows that
$M_C\leq M_A-1$. This concludes the proof of part (1). Part (2)
holds because for $l=1,\dots,k$, $\{(A_l)_{ij}\}$  is obtained upon
multiplication of $\{A_{ij}\}$ by  bounded sequences and by
permutations of  terms. \epf

\blemma \label{bigsum} Let $A=(\aij)_{i,j\in \NN}$ be a matrix with
entries in a unital ring $R$ such that $N(A)<\infty $. Then \be
\item $A=A_1+A_2+\cdots+ A_k,$  where $A_i\in \Gamma(R)A\Gamma(R)$, and
 $N(A_i)\leq 1$,  for all $i=1,\dots,k$.
\item  If in addition $R$ is a bornological algebra and $S\triqui\elli$ is a symmetric ideal such that $\{A_{ij}\}\in S(\N\times\N, R)$, then $\{(A_l)_{ij}\}\in S(\N\times\N,R)$, for all $l=1,\dots, k$.
\ee

\bpf
 Use Lemma \ref{elusive} and proceed by induction on $r(A)$  to write
\[A=\sum_1^k B_i,\ \text{ where } r(B_i)=1,\ c(B_i)\leq c(A),\ \text{ and } B_i\in
 \Gamma(R)A\Gamma(R).\]

Next apply the same procedure to each transpose matrix $B_i^t$
to get the decomposition in (1). The second statement follows from
the second part of Lemma \ref{elusive}. \epf

\elemma

\bprop \label{Nisone} Let $A=(\aij)_{i,j\in \NN}$ be a matrix with
entries in a ring $R$. Then $N(A)\leq 1$ if and only if
$A=\diag(\alpha)U_f$, where $f\in\emb$ and $\alpha\in R^\N$ are
defined as follows:
\[
f(j)=i\iff A_{ij}\neq 0 \quad \alpha(i)=\begin{cases}
A_{ij}, & \text{ if } i=f(j)\\
0, & \text{otherwise.}
\end{cases}
\]
\eprop \bpf

For $f$ and $\alpha$ as in the proposition, the $n^{th}$ column of
$A$ is
\begin{align*}
(\diag(\alpha)U_f)(e_n)=&\begin{cases}\alpha(n)e_{f(n)},&\text{ if } n\in\dom(f)\\ 0,& \text{otherwise.}    \end{cases}\\
 =&\begin{cases}A_{f(n)n}e_{f(n)},&\text{ if } n\in \dom(f)\\ 0,& \text{otherwise.}    \end{cases}
\end{align*}
\epf
\comment{
The following proposition is well known. We
include a proof since we have not been able to find one in the
literature.

\begin{prop}\label{lem:gamni}
Let $R$ be a ring. Then the set
\[
\{\diag(\alpha) U_f: f\in\emb\mbox{ and }\{\alpha_n:n\in\N\}\subset R\mbox{ finite }\}
\]
generates $\Gamma(R)$ as an abelian group.

\end{prop}
\bpf Let $X_R\subset\Gamma(R)$ be the set of elements of the
proposition. By \cite{biva}*{Lemma 4.7.1}, the map
$\phi:\Gamma\Z\otimes R\to \Gamma(R)$, $\phi(A\otimes
x)_{i,j}=A_{i,j}x$ is an isomorphism. This map sends the subgroup
generated by the elements $x\otimes r$ ($x\in X_\Z$, $r\in R$) onto
the subgroup generated by $X_R$. Thus it suffices to prove the
Proposition for $R=\Z$. Let $A\in\Gamma(\Z)$. By Lemma \ref{bigsum}
and Proposition \ref{Nisone}, $A=\sum_kA_k$, where $\Gamma\owns
A_k=\diag(\alpha^k)U_{f_k}$, for some $\Z$-valued sequence
$\alpha^k$. Besides, $A_k\in\Gamma(\Z)$ implies that the set
$\{\alpha^k_n:n\in\NN\}$ is finite. \epf

\begin{coro}
Let $\fA$ be a unital bornological algebra. Then  Karoubi's cone
$\Gamma(\fA)$ is a subalgebra of $\Gami(\fA)$.
\end{coro}
}
\bprop \label{gamma} Let $\fA$ be a unital bornological algebra,
$S\triqui\elli$ a symmetric ideal, and $I_{S(\fA)}\triqui\Gami(\fA)$
the ideal defined in equation \eqref{def:is}. Then
\begin{equation}\label{eq:mis}
I_{S(\fA)}=\{A=(\aij)_{i,j\in \NN}: \{\aij\}\in S(\N\times\N)\mbox{ and }
N(A)<\infty\}.
\end{equation}
\eprop

\bpf Let $D_S$ denote the set on the right hand side of equation
(\ref{eq:mis}). By Lemma \ref{bigsum} and Proposition \ref{Nisone},
a matrix $A$ belongs to $D_S$ if and only if $A=\sum A_k$, with
$A_k=\diag(\alpha_k)U_{f_k}\in D_S$. Further, we may choose $\alpha_k$ and $f_k$
such that $\supp(\alpha_k)=\ran(f_k)$. Under these conditions, $A_k\in D_S$ if
and only if $\alpha^k\in S$. This shows that
$A\in D_S$ if and only  $A\in I_S$.\epf

\begin{coro}
Let $\fA$ be a unital bornological algebra. Then  Karoubi's cone
$\Gamma(\fA)$ is a subalgebra of $\Gami(\fA)$.
\end{coro}

\begin{defi}\label{defi:gaminuni}
If $\fA$ is a not necessarily unital bornological algebra, and
$S\triqui\elli$ is a symmetric ideal, $I_{S(\fA)}$ is defined by
\eqref{eq:mis}.
\end{defi}

\begin{exa}\label{exa:minf}
Let
\[
c_f=\{\alpha\in\elli: \supp(\alpha)\text{ is finite }\}.
\]
Then
\begin{align*}
I_{c_f(\fA)}=M_\infty(\fA)=\{A:\exists n\in \NN \mbox{ such that } A_{ij}=0 \  \mbox{if either $i>n$ or $j>n$ } \}.
\end{align*}
We shall write $M_\infty=M_\infty\Z$.
\end{exa}

\brk\label{inid} {\rm Let $\fA$ be a unital bornological algebra, $I\triqui\Gami(\fA)$ a two-sided ideal and $T\in I$.
Then by Lemma
\ref{bigsum} and Remark \ref{Nisone}, we can write
\begin{equation}\label{laT}
T=\sum_{i=1}^n\diag(\alpha^i)U_{f_i}\text{ with } \diag(\alpha^i)U_{f_i}\in I,
\end{equation}
where $f_i\in\emb$ and $\alpha^i\in\elli(\fA)$. Similarly, if $R$ is a unital ring and $T\in I\triqui\Gamma(R)$, then we can also
write $T$ as in $\eqref{laT}$ but now with $\alpha^i$ such that the set
$\{\alpha^i_n:n\in \NN\}\subset R $ is finite.}
\erk

\section{The two-sided ideals of \topdf{$\Gami$}{Gami(C)} and those of \topdf{$\cB(\ell^2(\N))$}{B(H)}}\label{sec:ideals}

Calkin's theorem in \cite{calk}*{Theorem 1.6}), as restated by Garling in \cite{garling}*{Theorem ~1}, establishes  a bijective
correspondence between the set of proper two-sided ideals of
$\cB=\cB(\ell^2)$ and the set of proper symmetric ideals of $\elli$. Calkin defined this correspondence in
terms of the sequence of singular values of a compact operator. It
can also be described as follows: an ideal $J\triqui\cB$ is mapped
to the symmetric ideal
\begin{equation}\label{S(J)}
S(J)=\{\alpha\in\elli: \diag(\alpha)\in J\}.
\end{equation}
The inverse correspondence maps a symmetric ideal $S$ in $\elli$ to
the two-sided ideal
\begin{equation}\label{JS}
\cB\triangleright J_S=\langle\diag(\alpha):\alpha\in S\rangle
\end{equation}
We refer the reader to \cite{simon}*{Theorem 2.5} for further details.
Recall that, by another result of Calkin \cite{calk}*{Theorem 1.4},
the Calkin algebra $\cB/\cK$ is simple. On the other hand, it is
easily checked that $c_0\triqui\elli$ is maximal among proper
symmetric ideals. Thus, by mapping $\elli$ to $\cB$ we extend the
correspondence above to a bijection between the family of symmetric ideals of
$\elli$ and that of two-sided ideals of $\cB$. In Theorem
\ref{calkinforgami} below we show that Calkin's correspondence
carries over to ideals in $\Gami$. We will make use of the following
lemma.

\begin{lem} \label{poldecgami}  Let $\alpha\in\elli$, $f\in\emb$ and let $I\triqui\Gami$ a two-sided ideal.
Consider the operator
\[
T=\diag(\alpha)U_f.
\]
Then
\[
T\in I \iff |T|\in I.
\]
\end{lem}
\bpf We have
\[T^*T=U^*_f\ \diag(|\alpha|^2)\ U_f=\diag(f^\dagger_*(|\alpha|^2))=\diag(|f^\dagger_*(\alpha)|^2). \]
Therefore, $|T|=\diag(|f^\dagger_*(\alpha)|)$, and the polar decomposition of $T $ is $T=V|T|$, where
\[V=\diag(\nu_{\alpha})U_f,\] for
\begin{equation}\label{eq:nualfa}
 \nu_{\alpha}(n)=
\begin{cases}
0, &\text{if $\alpha(n)=0$}\\
\frac{\alpha(n)}{|\alpha(n)|}, &\text{otherwise.}
\end{cases}
\end{equation}

It is now clear that $V\in \Gami$. Thus $T\in I$ if and only if
$|T|\in I$, since $\Gami$ is a $*$-algebra and $|T|=V^*T$. \epf

\begin{thm}
\label{calkinforgami}
\item[i)] The map $S\mapsto I_S$ is a bijection between the set of symmetric ideals of $\elli$ and the set of two-sided ideals of $\Gami$. Its inverse maps an ideal $I\triqui \Gami$ to the symmetric ideal $S(I)$ defined as in \eqref{S(J)}.
\item[ii)] The map $J\mapsto J\cap\Gami$ is a bijection between the sets of two-sided ideals of
$\cB$ and those of $\Gami$. Its inverse maps an ideal $I\triqui\Gami$ to
the two-sided ideal of $\cB$ it generates.
\item[iii)] If $S\triqui\elli$ is a symmetric ideal, then $J_S\cap \Gami=I_S$.

\end{thm}

\bpf Let $I\triqui\Gami$; write $S=S(I)$. It is clear that
$I_S\subseteq I$. On the other hand, if $T=\diag(\alpha)U_f\in I$,
for some $\alpha\in \elli$ and $f\in  \emb$, then, by Lemma
\ref{poldecgami},
\[
\diag(f^\dagger_*(|\alpha|))=|T|\in I_S.
\]
Hence $T\in I_S$, again by Lemma \ref{poldecgami}. In view of Remark \ref{inid}, this implies that $I=I_S$. We have shown that $I_{S(I)}=I$. Let now $S\triqui\elli$ be a symmetric ideal. Then
\[
S\subset S(I_S)\subset S(J_S)\subset S,
\]
the last inclusion being due to Calkin's theorem. It follows that $S=S(I_S)$, completing the proof of part i).
Next, since the ideal $\langle I_S\rangle\triqui\cB(\ell^2)$ generated by $I_S$ is also generated by $\diag(S)$ we have $\langle I_S\rangle=J_S$, by Calkin's theorem. Now, again by Calkin's theorem,
\[
S\subset S(J_S\cap \Gami)\subset S(J_S)=S.
\]
Thus $J_S\cap\Gami=I_S$, by part i). We have proven part iii) and also shown that $\langle I_S\rangle \cap \Gami=I_S$. Moreover, by parts i) and
iii) we have
\[
\diag(\elli)\cap J_S=\diag(\elli)\cap J_S\cap \Gami=\diag(\elli)\cap I_S=\diag(S).
\]
It follows that $\langle J_S\cap \Gami\rangle=J_S$, which ends the proof.
\epf

It follows from Proposition \ref{gamma}, Example \ref{exa:minf} and Theorem \ref{calkinforgami} that
\[
I\cap\Gamc=M_\infty(\C)
\]
for every proper ideal $I\triqui\Gami$. The next proposition shows that in fact $M_\infty(\C)$ is the only proper ideal of $\Gamc$.

\begin{prop}
\label{idinminf}
Let $k$ be a field. Then $M_\infty(k)$ is the only proper two-sided ideal of $\Gamma(k)$.
\end{prop}

\bpf It is well known and easy to check that
$M_\infty(R)\triqui \Gamma(R)$ for any ring $R$. Let $I\ne 0$ be a
two-sided ideal of $\Gamma(k)$, and let $A\neq 0,\ A\in I$. If $i_0$
and $j_0$ are such that $A_{i_0j_0}\neq 0$, then
\begin{equation}
\label{eij}
E_{ij}=(A_{i_0j_0}) ^{-1}E_{ii_0}AE_{j_0j}\in I  \quad\forall i,j
\end{equation}
This shows that $M_\infty(k)\subseteq I$. Assume that the inclusion
is strict. Let $A\in I\setminus M_\infty(k)$. By Remark (\ref{inid}),
we may assume that $A=\diag(\alpha)U_f$ for $f\in\emb$ and
$\alpha\in k^\N$, where $\im (\alpha)= \{\alpha_n:n\in\NN\}$ is
finite and $\supp(\alpha)=\dom f\subset\N$ is infinite. Because $k$
is a field, we can multiply $A$ on the left by a diagonal matrix in
$\Gamma(k)$ to conclude that $U_f\in I$. But since $\ran(f)$ is
infinite, there are bijections $g:\N\to\dom(f)$ and $h:\ran(f)\to\N$
such that $hfg=1$. Hence $I$ must contain $1=U_hU_fU_g$. \epf

\section{\topdf{$\Gami$}{gami} as an infinite sum ring}
\label{sec:misce}
We begin this section by recalling some definitions from \cite{wag}
and \cite{biva}.
 A {\em sum ring} $(R,x_0,x_1,y_0,y_1)$ consists of a unital ring $R$ and elements $x_0,x_1,y_0,$ and $y_1\in R$ satisfying:
\begin{align}
\label{sumring}
y_0x_0=y_1x_1=1\\
x_0y_0+x_1y_1=1. \notag
\end{align}
If $R$ is a sum ring, the map
\begin{equation}
\label{defoplus} \oplus: R\times R\longrightarrow R, \text{ defined
by } r\oplus s=x_0ry_0+x_1sy_1,
\end{equation}
is a unital ring homomorphism. An {\em infinite sum ring} consists
of a sum ring $R$ equipped with a unital ring homomorphism
\begin{equation}
\label{defphi} \Phi:R \longrightarrow R \text{ such that }
r\oplus\Phi(r)=\Phi(r).
\end{equation}

The notion of infinite sum ring was introduced by Wagoner in \cite{wag}. He showed that if $R$ is unital, then the following
is an infinite sum ring:
\[
\Gamma^W(R):=\{A\in R^{\N\times\N}: A\cdot M_\infty R\subset
M_\infty R\supset M_\infty R\cdot A\}.
\]
We may regard $\Gamma^W(R)$ as a multiplier algebra of $M_\infty R$. One checks that a matrix $A\in\Gamma^W(R)$ if and only if every row and every column of $A$ has finite support. Let
\begin{equation}\label{map:fi}
f_i:\N\to\N,\ \ f_i(n)=2n-i \quad (i=0,1)
\end{equation}
The elements $x_i=U_{f^\dagger_i}$, $y_i=U_{f_i}$ satisfy conditions \eqref{sumring}. The homomorphism $\Phi$ is defined by
\begin{equation}\label{eq:elfinfi}
\Phi(A)=\sum_{k=0}^\infty x_1^kx_0Ay_0y_1^k=\sum_{k,i,j}^\infty
A_{ij}E_{2^{k+1}i+2^k-1,2^{k+1}j+2^k-1}.
\end{equation}
This map is well-defined because $(k,i)\mapsto 2^{k+1}i+2^k-1$ is
one-to-one; Wagoner showed in \cite{wag}*{pp 355} that it satisfies
\eqref{defphi}. Observe that the $x_i's$ and $y_i's$ are elements of
$\Gamma(R)$. It is not hard to check, and noticed in
\cite{biva}*{4.8.2}, that $\Phi(\Gamma(R))\subset\Gamma(R)$, whence
$\Gamma(R)$ is an infinite sum ring too. Now we remark that if $\fA$
is a bornological algebra, then
\[
\Gamma(\fA)\subset\Gami(\fA)\subset\Gamma^W(\fA).
\]
Furthermore, $\Phi$ also sends $\Gami(\fA)$ to itself. Thus if $\fA$ is unital, then $\Gami(\fA)$ is an infinite sum ring. We record this in the following
proposition.

\begin{prop}\label{prop:gaminfisuri}
Let $\fA$ be a unital bornological algebra, and
let $f_i$ be as in \eqref{map:fi} and $\Phi$ as in \eqref{eq:elfinfi}
Then \mbox{$(\Gami(\fA),
U_{f^{\dagger}_0},U_{f^{\dagger}_1},U_{f_0},U_{f_1},\Phi)$} is an
infinite sum ring.
\end{prop}
\begin{coro}\label{coro:hgami0}
Let $F:\ass\to\ab$ be a functor. Assume that the restriction of $F$
to unital $\C$-algebras is split-exact and $M_2$-stable. Then
$F(\Gami(\fA))=0$ for any unital bornological algebra $\fA$. If
furthermore $F$ is split exact on all $\C$-algebras, then
$F(\Gami(\fA))=0$ for any, not necessarily unital bornological
algebra $\fA$.
\end{coro}
\bpf Immediate from Proposition \ref{prop:gaminfisuri} and
\cite{friendly}*{Proposition 2.3.1}.\epf

\begin{exas}\label{exas:infinitesums}
Both Weibel's
homotopy algebraic $K$-theory \cite{kh} and periodic cyclic homology
\cite{cq} are $M_2$-stable and excisive on all $\Q$-algebras. Hence if $\fA$
is a bornological algebra, then
\[
KH_*(\Gami(\fA))=HP_*(\Gami(\fA))=0.
\]
Algebraic $K$-theory groups $K_n$ are split exact and $M_2$- stable
for $n\le 0$; the same is true of Karoubi-Villamayor $K$-groups
$KV_m$ for $m\ge 1$ (\cite{kv}*{Th\'eor\`eme 4.5}). Hence,
\[
K_n(\Gami(\fA))=KV_m(\Gami(\fA))=0\qquad (n\le 0, m\ge 1),
\]
again for all $\fA$. For positive $n$, the groups $K_n$ are still
split exact and $M_2$-stable on unital rings. The same is true of
both the Hochschild and cyclic homology groups $HH_n$ and $HC_n$ for $n\ge 0$; moreover these groups vanish for
$n\le -1$. Hence we have
\[
K_{n+1}(\Gami(\fA))=HH_n(\Gami(\fA))=HC_n(\Gami(\fA))=0 \qquad (n\ge
0)
\]
for any unital bornological algebra $\fA$.
\end{exas}

\section{The algebra \texorpdfstring{$\Gami(\fA)$}{Gami(fA)} as a crossed
product}\label{sec:crossp}
\numberwithin{equation}{section}
Let $2^\NN$ denote  the submonoid of idempotent elements of $\emb$
\[
2^\NN=\{p:p\in\emb\ \ p^2=p\} \subset \emb.
\]
 Note that if $p\in 2^\N$, then for $A=\ran(p)=\dom(p)$, we have
$U_p=\diag(\chi_A)$, the diagonal matrix on the sequence
\[
(\chi_{A})_n=\left\{\begin{matrix} 1 & n\in A\\ 0 & n\notin
A.\end{matrix}\right.
\]
We will often identify $p$, $U_p=\diag(\chi_A)$, and $\chi_A$.
Notice that
\begin{equation}\label{eq:f*p}
f_*(p)f=fp.
\end{equation}
The subgroup of $\Gamma$ generated by the image of $2^\N$ under $f\mapsto U_f$ is the subring
\[
\cP=\mspan\{U_p:p\in 2^\NN\}\subset\Gamma.
\]
We also consider the monoid rings $\Z[2^\NN]$ and $\Z[\emb]$, and
the two-sided ideals
\begin{gather}
I=\langle\{\chi_{A\sqcup B}-\chi_A-\chi_B:A,B\subset\NN,\ \ A\cap B=\emptyset\}\rangle\triqui \Z[2^\NN],\label{eq:I}\\
J=\langle\{\chi_{A\sqcup B}-\chi_A-\chi_B:A,B\subset\NN,\ \ A\cap
B=\emptyset\}\rangle\triqui \Z[\emb].\label{eq:J}
\end{gather}

Observe that $I$ and $J$ contain the element
\[
\chi_{A\cup B}-\chi_A-\chi_B-\chi_{A\cap B}
\]
for any pair of not necessarily disjoint subsets $A,B\subset \N$.

\begin{lem}\label{lem:presfh}
\item{i)} $\cP=\Z[2^\NN]/I$.
\item{ii)}$\fH=\Z[\emb]/J$
\item{iii)}If $\fA$ is a unital bornological algebra, then $\ell^\infty(\fA)\otimes_\cP\fH\cong\Gami(\fA)$ as $\cP$-bimodules.
\end{lem}
\bpf
It is clear that there are natural surjective ring homomorphisms
\begin{gather*}
\pi_1:\Z[2^\NN]/I\to \cP\text{ and}\\
\pi_2:\Z[\emb]/J\to \fH,
\end{gather*}
and a natural surjective $\cP$-bimodule homomorphism
\[\pi_3:\elli\otimes_\cP\fH\to\Gami.\]
Let $\xi=\sum_{j=1}^n \lambda_j\chi_{A_j}\in\Z[2^\NN]$ represent an element $\in\ker\pi_1$; for each subset $F\subset\{1,\dots,n\}$, let
$A_F=\bigcap_{j\in F}A_j\cap\bigcap_{j\notin F}A_j^c$. From $\pi_1(\xi)_{|A_F}=0$ we get
\[
A_F\ne\emptyset\Rightarrow\sum_{j\in F}\lambda_j=0.
\]
Next note that $\bigcup_{i=1}^nA_i=\sqcup_{F} A_F$; hence, modulo
$I$, we have
\begin{align*}
\xi\equiv &\sum_F\sum_{j=1}^n\lambda_j\chi_{A_j\cap A_F}\\
      = & \sum_F(\sum_{j\in F}\lambda_j)\chi_{A_F}=0.
\end{align*}
This proves i). In order to prove ii) we have to show that $\ker(\pi_2)=0$. Let $\xi=\sum_{j=1}^n\lambda_jf_j\in\Z[\emb]$ be a representative of an element in $\ker(\pi_2)$. Let $A_i=\dom f_i$, and let $A_F$ be as above; then $\xi\equiv \sum_F\xi\chi_{A_F}$. Hence we may assume that the $A_i$ are disjoint. Furthermore, upon replacing $\xi$ by $\xi\chi_{A_i}$, and elminating zero elements of $\emb$, we may assume that $A_1=\dots=A_n$. For each $j\in \NN$, we have
\begin{equation}\label{sum=0}
\sum_{i=1}^n\lambda_ie_{f_i(j)}=0.
\end{equation}
Let $K=\{f_i(j):i=1,\dots,n\}$; for each $k\in K$, let
$D_k=\{i:f_i(j)=k\}$. Then $D(j):=\{D_k\}_{k\in K}$ is a partition
of $\{1,\dots,n\}$, and $\sum_{i\in D_k}\lambda_i=0$. There is a
finite set $\cD$ of partitions arising in this way, since the number
of all partitions of $\{1,\dots,n\}$ is finite. For each $D\in\cD$,
let $J_D=\{j\in\NN: D(j)=D\}$. Then $\sqcup_{D\in\cD}J_D=\NN$, and
$\xi\equiv\sum_D \xi\cdot\chi_D$. Hence, upon replacing $\xi$ by
$\xi\chi_D$ if necessary, we may assume that $\cD$ has only one
element $D=\{D_1,\dots,D_r\}$. But $\xi\equiv\sum_i\chi_{D_i}\xi$,
so we further reduce to the case when $r=1$. This means that
$f_1=\dots=f_n$ and, by \eqref{sum=0}, $\sum_i\lambda_if_i$ is the
zero element of $\Z[\emb]$. We have proved ii). To prove iii) we
must show that $\pi_3$ is injective. Let
$\xi=\sum_{i=1}^n\alpha^{(i)}\otimes U_{f_i}\in \ker \pi_3$. Because
\[
\alpha\otimes U_f= \alpha \chi_{\supp(\alpha)\cap\ran{f}}\otimes
\chi_{\supp(\alpha)\cap\ran{f}}U_f\in \elli(\fA)\otimes_\cP\fH,
\]
we may assume that $\supp(\alpha_i)=\ran(f_i)$ ($i=1,\dots,n$).
Proceeding as above, we may assume that $\dom f_1=\dots=\dom f_n$. For each $j\in \NN$, we have
\begin{equation}\label{sum=01}
\sum_{i=1}^n\alpha^{(i)}_je_{f_i(j)}=0.
\end{equation}
Proceeding as above again, we may reduce to the case $f_1=\dots=f_n$. By \eqref{sum=01}, we have
$\sum_{i=1}^n\alpha^{(i)}=0$. Thus
\[
\xi=\sum_{i=1}^n\alpha^{(i)}\otimes U_{f_i}=(\sum_{i=1}^n\alpha^{(i)})\otimes U_{f_1}=0.
\]

\epf

\begin{rem}\label{rem:tight}
Given any monoid $M$, a representation of $M$ is
the same thing as module over the monoid ring $\Z[M]$. In view of
Lemma \ref{lem:presfh}, the modules over $\cP$ and
$\Gamma$ correspond to those representations of the inverse monoids
$2^\N$ and $\emb$ which are tight in the sense of Exel
(see \cite{ruy}*{Def. 13.1 and Prop. 11.9}).
\end{rem}

\begin{rem}\label{rem:gamar}
It was proved in \cite{biva}*{Lemma 4.7.1} that the map
\[
\psi:\Gamma\otimes R\to \Gamma(R),\ \ \psi(A\otimes x)_{i,j}=A_{ij}x 
\]
is an isomorphism. It follows from this that $\Gamma$ is flat as an abelian
group. Therefore the map $J\otimes R\to \Z[\emb]\otimes R$ is injective. 
Thus, by Lemma \ref{lem:presfh}, 
\[
\Gamma(R)=\Z[\emb]\otimes R/J\otimes R=R[\emb]/JR.
\]
Next observe that the inclusion $\cP\subset \Gamma$ is a split injection. Indeed the map 
$$\Gamma\to\cP,\ \ U_f\mapsto P_{\dom f}$$ 
is a left inverse.
It follows that if $R$ is any ring then the map $\psi:\cP\otimes R\to \cP(R):=\psi(\cP\otimes R)$ is an isomorphism. Thus using Lemma \ref{lem:presfh} and a similar argument as that given above for the case of $\Gamma$, one can show that
\[
\cP(R)=R[2^\N]/IR.
\] 
\end{rem}

Because $\emb$ is a monoid, if $\cA$ is a ring on which $\emb$
acts by ring endomorphisms we can form the \emph{crossed product}
$\cA\#\emb$. As an abelian group, $\cA\#\emb=\cA\otimes\Z[\emb]$
with multiplication given by
\begin{equation}\label{eq:cpform}
(a\# f)(b\# g)=af_*(b)\# fg.
\end{equation}
Here $\#=\otimes$ and $f_*(b)$ denotes the action of $f$ on $\emb$.
Now assume that the $\emb$-ring $\cA$ is also a $\cP$-algebra,
that is, it is a ring and a $\cP$-bimodule, and these operations
are compatible in the sense that
\[
(ap)b=a(pb)\ \ (a,b\in \cA,\ \ p\in\cP).
\]
Further assume that $\cA$ is central as a $\cP$-bimodule,
i.e. $pa=ap$ ($a\in \cA$, $p\in\cP$), and that
\[
pa=p_*(a)\qquad (p\in 2^\N).
\]
Under all these conditions, we say that $\cA$ is an
\emph{$\emb$-bundle} (cf. \cite{busex}*{Def. 2.10}). For
$J\triqui\Z[\emb]$ as in \eqref{eq:J}, we have
\begin{gather*}
\cA\#\emb\troqui \cA\#J=\mspan\{r\# j:r\in \cA, j\in J\}\mbox{ and }\\
\cA\#\emb\troqui L=\mspan\{rp\#h-r\#ph:r\in \cA, p\in\cP, h\in
\emb\}.
\end{gather*}
Set
\begin{equation}\label{eq:crossdef}
\cA\#_\cP\Gamma=\cA\#\emb/(L+\cA\#J).
\end{equation}
Thus, $\cA\#_\cP\Gamma=\cA\otimes_\cP \Gamma$ as left
$\cP$-modules, and the product is that induced by \eqref{eq:cpform};
we have
\begin{equation}\label{eqprodcp}
(a\# U_f)(b\# U_g)=a f_*(b)\# U_{fg}\in \cA\#_\cP\Gamma.
\end{equation}

\begin{prop}\label{prop:cpg}
Let $\fA$ be a bornological algebra. The map
\begin{equation}\label{map:cp}
\elli(\fA)\#_\cP\fH\to \Gami(\fA),\quad \alpha\# U_f\mapsto \diag(\alpha) U_f
\end{equation}
is an isomorphism of $\cP$-algebras. If $S\triqui\elli$ is a
symmetric ideal, then \eqref{map:cp} sends $S(\fA)\#_\cP\fH$
isomorphically onto $I_{S(\fA)}\triqui\Gami(\fA)$.
\end{prop}
\bpf Assume first that $\fA$ is unital. Then the map \eqref{map:cp}
is the same as that of Lemma \ref{lem:presfh}(iii). Hence, it
is bijective. By \eqref{eqcov} and \eqref{eqprodcp},
it is an algebra homomorphism. This proves the first assertion in
the unital case; the second is immediate from the fact that
\eqref{map:cp} is bijective and maps $S(\fA)\#_\cP\fH$ onto
$I_{S(\fA)}$. For not necessarily unital $\fA$, write $\tilde{\fA}$
for its unitalization as a bornological algebra. We have an exact
sequence
\begin{equation}\label{seq:unitalis}
0\to S(\fA)\to S(\tilde{\fA})\to S\to 0.
\end{equation}
Observe that the inclusion $\C\subset\tilde{\fA}$ induces a
$\cP$-module homomorphism $S\to S(\tilde{\fA})$ which splits the
sequence \eqref{seq:unitalis}. Hence we get an exact sequence
\[
0\to S(\fA)\#_{\cP}\fH\to S(\tilde{\fA})\#_{\cP}\fH\to
S\#_{\cP}\fH\to 0.
\]
Combining this sequence with the unital case of the proposition, we
obtain an isomorphism
\[
S(\fA)\#_{\cP}\fH\iso \ker(I_{S(\fA)}\to I_S)=I_{S(\fA)}.
\]
\epf

\section{Homotopy invariance}\label{sec:hit}

\numberwithin{equation}{subsection}

\subsection{Crossed products by the Cohn ring}
The following two elements of $\emb$ will play a central role in
what follows
\begin{gather*}\label{map:cohn}
s_i:\N\to\N\quad (i=1,2)\\
s_i(m)=2m+i-1.
\end{gather*}
We have the following relations
\begin{equation}\label{eq:cohn}
s_i^\dagger s_j=\delta_{i,j}\qquad i=1,2.
\end{equation}
 Following standard conventions, if $\nu$ is a word of length $l$
on $\{1,2\}$, we write $s_\nu=s_{\nu_1}\cdots s_{\nu_l}$ and
$s^\dagger_\nu=(s_\nu)^\dagger$. Thus for the empty word we have
$s_\emptyset=s^\dagger_\emptyset=1$. Observe that if $\mu$ is of length $l$ then
\begin{equation}\label{esemu}
s_\mu(n)=2^ln+\sum_{i=1}^l(\mu_i-1)2^{i-1}.
\end{equation}
Put
\[
W_2^l=\{\text{ words of length $l$ on $\{1,2\}$}\},\quad
W_2=\bigcup_{l=0}^\infty W_2^l.
\]
We write
\[
\cM_2=\{s_\mu (s_\nu)^\dagger:\mu,\nu\in W_2\}.
\]
Thus $\cM_2\subset\emb$ is the inverse submonoid generated by the
$s_i$. Its idempotent submonoid is
\[
E(\cM_2)=\{s_\nu (s_\nu)^\dagger:\nu\in W_2\}.
\]
One checks, using \eqref{esemu} that $s_\mu s^\dagger_\nu=s_{\mu'}s^\dagger_{\nu'}$ if and only if $\mu=\mu'$ and $\nu=\nu'$. It follows that $\cM_2$ is the universal inverse monoid on
generators $s_1,s_2$ subject to the relations \eqref{eq:cohn}. Write
\[
C_2=\Z[\cM_2]\supset \cP_2=\Z[E(\cM_2)].
\]
The algebra $C_2$ is the \emph{Cohn ring} on two generators
(\cite{cohn}). The assignment
\[
E_{s_\mu(1),s_\nu(1)}\mapsto
s_\mu(1-\sum_{i=1}^2s_is^\dagger_i)s^*_\nu.
\]
defines an isomorphism between $M_\infty$ and the ideal of $C_2$
generated by $1-\sum_{i=1}^2s_is^\dagger_i$. We shall identify each
element of $M_\infty$ with its image in $C_2$. If $\fA$ is a
bornological algebra and $S\triqui\elli$ is a symmetric ideal, then
we can consider the action of $\cM_2$ on $S(\fA)$ coming from
restriction of the $\emb$ action, and form the crossed product
$S(\fA)\#\cM_2$. Recall from Section \S\ref{sec:crossp} that
$S(\fA)\#\cM_2=S(\fA)\otimes_\Z\Z[\cM_2]$ equipped with the product
\eqref{eq:cpform}. Put
\[
S(\fA)\#_{\cP_2}C_2=S(\fA)\#\cM_2/\langle\alpha p\#f-\alpha\#pf:p\in
E(\cM_2), f\in \cM_2\rangle.
\]
As a vector space, $S(\fA)\#_{\cP_2}C_2=S(\fA)\otimes_{\cP_2}C_2$;
the product is defined as in \eqref{eq:cpform}. We have an algebra
homomorphism

\begin{equation}\label{map:rhocohn}
\rho:S(\fA)\#_{\cP_2}C_2\to I_{S(\fA)},\ \
\rho(\alpha\#f)=\diag(\alpha)U_f.
\end{equation}

\begin{lem}\label{lem:rhocohn}
The map \eqref{map:rhocohn} is injective.
\end{lem}
\bpf
Any nonzero element $x\in C_2$ can be written as a finite sum of nonzero terms
\begin{equation}\label{elequis}
x=\sum_{\mu,\nu}\alpha_{\mu,\nu}\#s_\mu s^\dagger_\nu.
\end{equation}
Let $l$ be the maximum length of all the multi-indices $\mu$ appearing in the
expression above. Remark that we may rewrite \eqref{elequis} as another finite sum
\begin{equation}\label{xrescrito}
x=\sum_{i,j}x_{i,j}\# E_{i,j}+\sum_{l(\mu)=l}\beta_{\mu,\nu}\#
s_{\mu}s^\dagger_{\nu}.
\end{equation}
such that
\begin{equation}\label{eq:xij}
x_{i,j}\ne 0\Rightarrow i<2^l.
\end{equation}
Indeed this follows from \eqref{esemu} and from the
identities
\begin{align*}
s_\mu s^\dagger_\nu=&s_\mu(1-\sum_{i=1}^2s_is_i^\dagger)s^\dagger_\nu+\sum_{i=1}^2s_{\mu i}s_{\nu i}^{\dagger}\\
                   =&E_{\mu(1),\nu(1)}+\sum_{i=1}^2s_{\mu i}s^{\dagger}_{\nu i}.
\end{align*}
Suppose that the element \eqref{xrescrito} is in $\ker\rho$. Observe
that $\rho(\chi_{\{i\}}\otimes E_{i,j})=E_{i,j}$. Hence, we have
\begin{equation}\label{vacero}
0=\sum_{i,j}x_{i,j}E_{i,j}+\sum_{l(\mu)=l,\nu}\diag(\beta_{\mu,\nu})U_{s_\mu}U_{s_\nu}^*.
\end{equation}
But by \eqref{esemu}, for $\mu$ as in \eqref{vacero}, we have
\[
\ran (U_{s_\mu}U_{s_\nu}^*)=\mspan\{e_n: n=2^l
m+\sum_{i=1}^l(\mu_i-1)2^{i-1}\ \ m\in\N\}.
\]
This together with \eqref{eq:xij} imply that each of the summands of \eqref{vacero} vanishes. Thus
\[
x_{i,j}=0 \text{ and } \diag(\beta_{\mu,\nu})U_{s_\mu}U_{s_\nu}^*=0
\]
for all $i,j$ and all $\mu$ and $\nu$ in \eqref{eq:xij}. Hence,
\[
\emptyset=\supp
\beta_{\mu,\nu}\cap(2^l\N+\sum_{i=1}^l(\mu_i-1)2^{i-1}) =\supp(s_\mu
s^\dagger_\mu)_*(\beta_{\mu,\nu}).
\]
It follows that $\beta_{\mu,\nu}\#s_\mu s^\dagger_\nu=0$ and therefore the element \eqref{xrescrito} must be zero, finishing
the proof.
\epf
\begin{rem}\label{rem:complecohn}
Let $S\triqui\elli$ be a nonzero symmetric ideal and let $c_f$ be as in Example \ref{exa:minf}. Then $S$ contains $c_f$ and thus if we identify $S\#_{\cP_2}C_2$ with its
image in $I_S$, we have
\[
I_S\supset S\#_{\cP_2}C_2\supset c_f\#_{\cP_2}C_2=M_\infty.
\]
In particular the completion of $c_0\#_{\cP_2}C_2$ with respect to the operator norm in $\cB(\ell^2)$
coincides with the completion of $M_\infty\C$ and of $I_{c_0}$; it is the ideal $\cK=J_{c_0}$ of compact
operators. Similarly, for $p\ge 1$ the completion of $\ell^p\#_{\cP_2}C_2$ for the $p$-Schatten norm
$||T||_p=Tr(|T|^p)$ coincides with that of $I_{\ell^p}$; it is the Schatten ideal $\Ell^p$.
\end{rem}

\subsection{The Cohn ring and homotopy invariance}
Let $\V$ be a bornolo- gical vector space, $T$ a compact Hausdorff
topological space, $X$ a metric space, and $1\ge\lambda>0$. Put
\begin{gather*}
C(T,\V)=\{f:T\to\V\text{ continuous}\},\\
H^\lambda(X,\V)=\{f:X\to\V \ \ \lambda-\text{H\"older continuous}\}.
\end{gather*}
We refer the reader to \cite{cmr}*{\S 2.1.1 and \S 3.1.4} for the definitions of
continuity and H\"older continuity in the bornological setting, as well as for
those of the canonical uniform bornologies that the above algebras carry.

Let $S\triqui\elli$ be a symmetric ideal and $\fA$ a bornological
algebra. We have a natural inclusion
\[
\inc:\fA\subset S(\fA), a\mapsto (a,0,0,\dots).
\]

\begin{lem}\label{lem:c2} (cf. \cite{cmr}*{Lemma 3.26})
Let $F:\ass\to\ab$ be a split-exact, $M_2$-stable functor, $\fB$ a
bornological algebra, $\ev_t:C([0,1],\fB)\to \fB$ the evaluation
map, and $0<\lambda\le 1$.

\item[i)]
\[
 F\left(C([0,1],\fB)\overset{\ev_t}\to \fB\overset{\inc}\to c_0(\fB)\overset{-\#1}\to
 c_0(\fB)\#_{\cP_2}C_2\right)
\]
is independent of $t$.
\item[ii)]
If $p>1/\lambda$, then
\[
F\left(H^\lambda([0,1],\fB)\overset{\ev_t}\to \fB\overset{\inc}\to
\ellp(\fB)\overset{-\#1}\to \ellp(\fB)\#_{\cP_2}C_2\right)
\]
is independent of $t$.
\end{lem}
\begin{proof}
Let $S$ be either $c_0$ or $\ell^p$. In the first case, put
$\fB[0,1]=C([0,1],\fB)$; in the second, let $\lambda>1/p$ and set
$\fB[0,1]=H^\lambda([0,1],\fB)$. Let
\[
\Z_{\ge 0}\times\Z_{\ge 0}\supset X=\{(l,k):k\le 2^l-1\}.
\]
Let $\phi_+,\phi_-,\phi_0^2$ and $\phi_-^2$ be the homomorphisms
$\fB[0,1]\to \elli(X,\fB)$ defined in the proof of \cite{cmr}*{Lemma
3.26}. One checks that $(\phi_+,\phi_-)$ and $(\phi_0^2,\phi_-^2)$
are quasi-homomorphisms $\fB[0,1]\to S(X,\fB)$. Furthermore, it is
shown in loc. cit. that there are elements $V,\bar{V}\in\emb(X)$
such that for
\[
\inc_{0,0}:\fB\to S(X,\fB),\ \
\inc_{0,0}(a)_{l,k}=a\delta_{l,0}\delta_{k,0}
\]
we have
\begin{multline}\label{eq:ralf}
F(\inc_{0,0}\circ\ev_0)-F(\inc_{0,0}\circ\ev_1)=(F(\bar{V}_*)-1)
F(\phi_-,\phi_+)\\
+(F(V_*)-1) F(\phi^2_0,\phi^2_-).
\end{multline}
Consider the bijection $\psi:X\to \N$
\begin{equation}\label{psi}
\psi(l,k)=2^l+k.
\end{equation}
Let $s_1,s_2$ be the generators \eqref{map:cohn} of
$C_2$. Let $v,\bar{v}\in\emb$ be the conjugates of $V$
and $\bar{V}$ under $\psi$. One checks that, for $\rho$ as in \eqref{map:rhocohn},
we have
\begin{gather}
\bar{v}=s_2\mbox{ and}\label{barvs2}\\
U_{v}=\rho(1-s_1s_1^\dagger-s_2s_2^\dagger+s_2s_1^\dagger+s_1s_2^\dagger).\label{velotro}
\end{gather}
Now recall that $C_2=\Z[\cM_2]$ and write $*:C_2\to C_2$ for the involution induced by $\dagger$.
It follows from \eqref{velotro} that the element
\begin{equation}\label{elefe}
C_2\owns f=1-s_1s_1^\dagger-s_2s_2^\dagger+s_2s_1^\dagger+s_1s_2^\dagger
\end{equation}
satisfies $f^*f=1$. Hence if $g$ is any of $1\#
s_2,1\#f\in\elli(\tilde{\fB})\#C_2$, we have an algebra homorphism
\[
\conj(g):S(\fB)\# C_2\to S(\fB)\# C_2,\ \ x\mapsto gxg^*.
\]
Moreover, because $F$ is $M_2$-stable by assumption and
$S(\fB)\#C_2$ is an ideal in $\elli(\tilde{\fB})\#C_2$,
$F(\conj(g))$ is the identity (\cite{friendly}*{Proposition 2.2.6}).
Let $\phi'^2_0$, $\phi'^2_-$, $\phi'_+$ and $\phi'_-$ be the maps
$\fB[0,1]\to S(\fB)$  obtained from $\phi^2_0$, $\phi^2_-$,
$\phi_+$, and $\phi_-$ after conjugating with $U_\psi$. Then
\eqref{eq:ralf} gives the identity
\begin{multline*}
F((\inc\ev_0)\#1)-F((\inc\ev_1)\#1)=\\
(F(\conj(1\#s_2))-1) F(\phi'_-,\phi'_+)+ (F(\conj(1\#f))-1)
F(\phi^2_0,\phi^2_-)=0.
\end{multline*}
We have proved that $F((\inc\circ\ev_0)\#1)=F((\inc\circ\ev_1)\#1)$.
The proposition now follows from the fact that if $t\in [0,1]$ then
$\ev_t$ and $\ev_0$ are linearly homotopic.
\end{proof}

\begin{rem}
The key property of $C_2$ used in the proof of Lemma \ref{lem:c2} is
that it contains the elements \eqref{barvs2} and \eqref{elefe}. In
fact it is not hard to check that they generate $C_2$ as a ring.
Hence taking crossed product with $C_2$ may be regarded as
the smallest construction which makes the proof
given above work.
\end{rem}

\begin{rem}\label{rem:stabletop}
If $\fA$ is a $C^*$-algebra, then $c_0(\fA)=c_0\sotimes\fA$ is the
spatial $C^*$-algebra tensor product. The inclusion $c_0\subset
I_{c_0}\subset\cK$ is equivariant for the action of $\emb$, and so
we get a map $c_0(\fA)\#_{\cP_2}C_2\to \fA\sotimes\cK$. Composing
the latter with the inclusion $\fA\to c_0(\fA)\#_{\cP_2}C_2$ of
Lemma \ref{lem:c2} we obtain the map $\iota:\fA\to\fA\sotimes\cK$,
$a\mapsto a\sotimes E_{1,1}$. Hence, the lemma implies that if
$F:\ass\to\ab$ is split-exact and $M_2$-stable, then, for every
$C^*$-algebra $\fB$, the map
\[
 F\left(C([0,1],\fB)\overset{\ev_t}\to \fB\overset{\iota}\to \fB\sotimes\cK\right)
\]
is independent of $t$. One can use this to prove that $F$ is homotopy invariant on stable $C^*$-algebras, thus obtaining a weak version of Higson's homotopy invariance theorem \cite{hig}*{Theorem 3.2.2}. Indeed it suffices to show that
$F(\iota)$ is injective if $\fB=\fA\sotimes\cK$, and this follows from the fact that there is a map $\cK\sotimes\cK\to M_2\cK$ (in fact an isomorphism) such that the following diagram commutes
\begin{equation}\label{diag:stablek}
\xymatrix{\cK\sotimes\cK\ar[r]&M_2\cK\\
\cK.\ar[u]^\iota\ar[ur]_{E_{1,1}}}
\end{equation}
Next suppose that $\fB$ is any bornological algebra. Write
$\hotimes$ for the projective tensor product. For each $p\ge 1$ we
have a map $\ell^p\hotimes\fB\to \ell^p(\fB)$. This map is an
isomorphism if $p=1$; using this isomorphism as above, we obtain a
map
\[
\ell^1(\fA)\#_{\cP_2}C_2\to \fA\hotimes\Ell^1.
\]
In general $\ell^p\hotimes\fA\to\ell^p(\fA)$ is not an isomorphism.
Note, however, that for every $p\ge 1$, the quotient
$\ell^p(\fA)/\ell^1(\fA)$ is a nilpotent ring. Assume that the
functor $F$ is \emph{strongly nilinvariant} in the sense that if
$f:A\to B$ is a homomorphism with nilpotent kernel, and such that
$f(A)\triqui B$ and $B/f(A)$ is nilpotent, then $F(f)$ is an
isomorphism. Then $F(\ell^1(\fA)\#_{\cP_2}C_2)\to
F(\ell^p(\fA)\#_{\cP_2}C_2)$ and  $F(\fA\hotimes\Ell^1)\to
F(\fA\hotimes\Ell^p)$ are isomorphisms for all $p\ge 1$. Moreover we
also have a commutative diagram
\begin{equation}\label{diag:stablell}
\xymatrix{\Ell^1\hotimes\Ell^1\ar[r]&M_2\Ell^1\\
\Ell^1\ar[u]^\iota\ar[ur]_{E_{1,1}}}
\end{equation}
Let $\bass$ be the category of bornological algebras and bounded
homomorphisms. Using Lemma \ref{lem:c2} together with diagram
\eqref{diag:stablell} above and proceeding as before, one shows that
if $F$ is split-exact, $M_2$-stable, and strongly nilinvariant, then
the functor
\[\bass\to\ab,\ \ \fA\mapsto F(\fA\hotimes\Ell^1),\]
is invariant under H\"older-continuous homotopies. This gives a
(weak) borno- logical version of \cite{cot}*{Theorem 6.1.6}. Observe
that the stability properties \eqref{diag:stablek} and
\eqref{diag:stablell} play a crucial role in the arguments above. We
do not have an analogue stability result for the uncompleted
algebras $c_0(\fA)\#_{\cP_2}C_2$ and $\ell^1(\fA)\#_{\cP_2}C_2$. In
the next subsection we shall prove a version of stability for
crossed products with $\Gamma$. This will enable us to prove a
homotopy invariance theorem in the following subsection.
\end{rem}

\subsection{Stability}

\begin{lem}\label{lem:gamastable}
\item[i)] There is a natural isomorphism $\Gamma(\N\sqcup\N)\cong M_2\Gamma$.
\item[ii)]Let $\fA$ be a bornological algebra and $S\triqui\elli$ a symmetric ideal. Then  $I_{S(\N\sqcup\N,\fA)}\cong M_2I_{S(\fA)}$.
\end{lem}
\begin{proof}
Let $p_1,p_2\in\emb(\N\sqcup\N)$ be the inclusions of each of the
copies of $\N$. If $f\in\emb(\N\sqcup\N)$, then $p_ifp_j$ identifies
in the obvious way with an element $f_{i,j}\in\emb$. One checks that
the map
\[
\emb(\N\sqcup\N)\to M_2\Gamma,\quad f\mapsto (U_{f_{ij}})
\]
is multiplicative. Hence it induces a homomorphism
\[\Z[\emb(\N\sqcup\N)]\to M_2\Gamma.\]
One checks further that this map kills the ideal \eqref{eq:J}, and
thus descends to a homomorphism
\begin{equation}\label{map:gamam2}
\phi:\Gamma(\N\sqcup\N)\to M_2\Gamma,\quad
\phi(a)_{ij}=U_{p_i}aU_{p_j}.
\end{equation}
Observe that $E_{i,j}U_f$ is in the image of \eqref{map:gamam2} for
all $f\in\emb$. It follows that \eqref{map:gamam2} is surjective.
Moreover because $U_{p_1},U_{p_2}$ are orthogonal idempotents with
$U_{p_1}+U_{p_2}=1$, $a\in\Gamma(\N\sqcup\N)$ is zero if and only if
$U_{p_i}aU_{p_j}=0$ for $1\le i,j\le 2$. Hence \eqref{map:gamam2} is
an isomorphism; this proves part i). To prove part ii) one begins by
observing that the assignment $\alpha\mapsto (\alpha p_1, \alpha
p_2)$ defines isomorphisms $S(\N\sqcup\N)\iso S(\N)\oplus S(\N)$ and
$\cP(\N\sqcup\N)\iso \cP(\N)\oplus\cP(\N)$. Next, note that if we
regard $M_2\Gamma$ as a $\cP\oplus\cP$-module via the diagonal
inclusion, we have an isomorphism of abelian groups
\begin{gather*}
(S(\fA)\oplus S(\fA))\otimes_{\cP\oplus\cP}M_2(\Gamma)\cong
M_2(S(\fA)\#_\cP\Gamma)\\
(\alpha_1,\alpha_2)\otimes x\mapsto \sum_{1\le i,j\le 2}\alpha_i\#
x_{i,j}\otimes E_{i,j}.
\end{gather*}
Finally one checks that the algebra homomorphism
\begin{gather*}
S(\N\sqcup\N,\fA)\#_{\cP(\N\sqcup\N)}\Gamma(\N\sqcup\N)\to
M_2(S(\fA)\#_\cP\Gamma)\\
\alpha\# x\mapsto\sum_{1\le i,j\le 2}\alpha p_i\#U_{p_i}xU_{p_j}\otimes
E_{i,j}
\end{gather*}
coincides with the following composite of isomorphisms of abelian
groups
\begin{align*}
S(\N\sqcup\N,\fA)\#_{\cP(\N\sqcup\N)}\Gamma(\N\sqcup\N)\cong &
(S(\fA)\oplus S(\fA))\otimes_{\cP\oplus\cP}M_2(\Gamma)\\
 \cong &M_2(S(\fA)\#_\cP\Gamma).
\end{align*}
\end{proof}

Let $\fA$ be a bornological algebra and let
$\iota:\elli(\fA)\to\elli(\N\times\N,\fA)$ be the inclusion
\[
\iota(\alpha)(m,n)=\alpha_m\delta_{1,n}.
\]
Also let
$S\triqui\elli$ be a symmetric ideal; put

\begin{gather}\label{map:jota}
\jmath:S(\fA)\#_{\cP}\Gamma\to S(\N\times\N,\fA)\#_{\cP(\N\times\N)}\Gamma(\N\times\N)\\
\jmath(\alpha\#U_f)=\iota(\alpha)\#(U_{f\times
\chi_{\{1\}}}).\nonumber
\end{gather}

\begin{prop}\label{prop:semigami}
Let $\fA$ be a bornological algebra and $S\triqui\elli$ a symmetric
ideal. Then any $M_2$-stable functor $F:\ass\to\ab$ sends the map
$\jmath$ of \eqref{map:jota} to a split monomorphism.
\end{prop}
\begin{proof}
Choose a bijection $\N\times\N\to \N\sqcup\N$ sending
$\N\times\{1\}$ bijectively onto the first copy of $\N$. This
bijection induces an isomorphism
\[
S(\N\times\N,\fA)\#_{\cP(\N\times\N)}\Gamma(\N\times\N)\overset{\cong}\longrightarrow
S(\N\sqcup\N,\fA)\#_{\cP(\N\sqcup\N)}\Gamma(\N\sqcup\N).
\]
Composing this map with the isomorphism of Lemma
\ref{lem:gamastable}, we obtain an isomorphism which fits into a
commutative diagram
\[
\xymatrix{S(\N\times\N,\fA)\#_{\cP(\N\times\N)}\Gamma(\N\times\N)
\ar[r]^(.6)\sim &M_2(S(\fA)\#_{\cP}\Gamma)\\
S(\fA)\#_{\cP}\Gamma\ar[u]^{\jmath}\ar[ur]_{E_{1,1}\otimes-}}
\]
This concludes the proof.
\end{proof}

\subsection{A homotopy invariance theorem}
Let $f_0,f_1:\fA\to\fB$ be homomorphisms of bornological algebras
and $0<\lambda\le 1$. A \emph{$\lambda$-H\"older continuous
homotopy} from $f_0$ to $f_1$ is a homomorphism $H:\fA\to
H^\lambda([0,1],\fB)$ such that $\ev_i H=f_i$ ($i=0,1$). We say that
a functor $F:\bass\to\ab$ is \emph{invariant under
$\lambda$-H\"older homotopies} if it maps $\lambda$-H\"older
homotopic homomorphisms to equal maps.

\begin{thm}\label{thm:htpy}
Let $F:\ass\to\ab$ be a split-exact, $M_2$-stable functor.
\item[i)] The functor
\[
\bass\to\ab, \fB\mapsto F(I_{c_0(\fB)})
\]
is invariant under continuous homotopies.
\item[ii)] If $1\ge\lambda>0$ and $p>1/\lambda$, then the functor
\[
\bass\to\ab, \fB\mapsto F(I_{\ell^p(\fB)})
\]
is invariant under $\lambda$-H\"older homotopies.
\end{thm}
\begin{proof}
Let $\fA$ be a bornological algebra. We adopt the notations of the
proof of Lemma \ref{lem:c2}. Thus $S$ will be either of $c_0$ or
$\ell^p$, and $\fA[0,1]$ will stand for $C([0,1],\fA)$ in the first
case, and for $H^\lambda([0,1],\fA)$ in the second. By the argument
of the proof of Lemma \ref{lem:c2} applied to the functor
\begin{equation}\label{elG}
G=F(S(-)\#_\cP\Gamma),
\end{equation}
 we have the following identity
\begin{multline}\label{eq:ralfeado}
G(\inc)\left(G(\ev_0))-G(\ev_1)\right)=(G((s_2)_*)-1)
G(\phi'_-,\phi'_+)\\
+(G(f_*)-1) G({\phi'}^2_0,{\phi'}^2_-).
\end{multline}
Now if $h\in\emb$ then $G(h_*)$ is the result of applying $F$ to the map
\[
S(h_*)\#_\cP\Gamma:S(S(\fA))\#_\cP\Gamma\to S(S(\fA))\#_\cP\Gamma.
\]
Here the crossed product is taken with respect to the action on the
external $S$. In addition, we consider the action of $\Gamma$ on the
inner $S$ and take the crossed product again; we write
$(S(S(\fA))\#_\cP\Gamma)\#_\cP\Gamma$ for the resulting algebra. We
have an inclusion

$$\inc'=-\#1:S(S(\fA))\#_\cP\Gamma\subset (S(S(\fA))\#_\cP\Gamma)\#_\cP\Gamma$$

and a commutative diagram
\[
\xymatrix{S(S(\fA))\#_\cP\Gamma\ar[r]^{S(h_*)\#_\cP\Gamma}\ar[d]_{\inc'}&
S(S(\fA))\#_\cP\Gamma\ar[d]^{\inc'}\\
(S(S(\fA))\#_\cP\Gamma)\#_\cP\Gamma\ar[r]_{\conj(1\#U_h)}&
(S(S(\fA))\#_\cP\Gamma)\#_\cP\Gamma}
\]
Because $F$ is $M_2$-stable, $F(\conj(1\#U_h))$ is the identity map, since
\[S(S(\fA))\#_\cP\Gamma)\#_\cP\Gamma\triqui (\elli(\elli(\fA))\#_\cP\Gamma)\#_\cP\Gamma\owns 1\#U_h .\]
Hence, by \eqref{eq:ralfeado},
\begin{multline}
F(\inc'(S(\inc)\#\Gamma)))
(F(S(\ev_0)\#\Gamma)-F(S(\ev_1)\#\Gamma)=\\
F(\inc')(G((s_2)_*)-1) G(\phi'_-,\phi'_+)
\\+F(\inc')(G(f_*)-1) G({\phi'}^2_0,{\phi'}^2_-)=0.
\end{multline}
We have to show that
\begin{equation}\label{map:claiminj}
F(\inc'(S(\inc)\#\Gamma))
\end{equation}
is injective. Observe that we have a natural isomorphism
\begin{equation}\label{multiso}
\mu:S(S(\fA))\iso S(\N\times\N,\fA),\ \
\mu(\alpha)_{m,n}=(\alpha_n)_m.
\end{equation}
For $h\in\emb$ the isomorphism \eqref{multiso} transforms $S(h_*)$
into the action of $1\times h\in\emb(\N\times\N)$, and $h_*S$ into
that of $h\times 1$. Hence we have a map
\begin{gather*}
\inc":(S(S(\fA))\#_\cP\Gamma)\#_\cP\Gamma\to
S(\N\times\N)\#_{\cP(\N\times\N)}\Gamma(\N\times\N)\\
\inc"(\alpha\#U_g\#U_h)=\mu(\alpha)\#(U_{g\times h}).
\end{gather*}
Observe that the composite
\[
\inc"\inc' (S(\inc)\#\Gamma)=\jmath
\]
is the map of \eqref{map:jota}. By Proposition \ref{prop:semigami},
this implies that the map \eqref{map:claiminj} is injective,
concluding the proof.
\end{proof}

\section{\topdf{$K$}{K}-theory}\label{sec:K}

\subsection{Homotopy algebraic \topdf{$K$}{K}-theory}

Let $0<p\le\infty$. Put
\[
\ell^{p-}=\bigcup_{q<p}\ell^q.
\]

For $0<p<\infty$ we also consider
\[
\ell^{p+}=\bigcap_{q>p}\ell^q.
\]
We say that a functor $F:\bass\to\ab$ is \emph{H\"older homotopy
invariant} if it is invariant under $\lambda$-H\"older homotopies
for all $0<\lambda\le 1$. Recall from \cite{cmr}*{\S2} that a
bornological algebra is called a \emph{local Banach algebra} if it
is a filtering union of Banach subalgebras. Similarly we say that a
bornological algebra is a \emph{local $C^*$-algebra} if it is a
filtering union of $C^*$-subalgebras. If $\fA=\cup_\lambda
\fA_\lambda$ and $\fB=\cup_\mu \fB_\mu$ are local $C^*$-algebras, we
define their spatial tensor product as the algebraic colimit of the
spatial tensor products $\fA_\lambda\sotimes \fB_\mu$;
$\fA\sotimes\fB=\colim_{\lambda,\mu}\fA_\lambda\sotimes\fB_\mu$. For
the projective tensor product of bornological spaces (and of
bornological algebras) see \cite{cmr}*{\S 2.1.2}. In the next
theorem and elsewhere we write $KV$ for Karoubi-Villamayor's
$K$-theory.

\begin{thm}\label{thm:kh}
Let $S$ be one of $\ell^p$, $\ell^{p+}$ ($0<p<\infty$) or
$\ell^{p-}$ ($0<p\le\infty$).
\item[i)] The functor $\bass\to \ab$, $\fA\mapsto KH_*(I_{\ell^1(\fA)})$
is H\"older homotopy invariant and we have
$KH_*(I_{S(\fA)})=KH_*(I_{\ell^1(\fA)})$ for all $S$ as above.

\item[ii)] For every bornological algebra $\fA$
\[
KH_n(I_{\ell^1(\fA)})=\left\{\begin{matrix}KV_n(I_{\ell^1(\fA)})& n\ge
1\\
K_n(I_{\ell^1(\fA)})& n\le 0.\end{matrix}\right.
\]

\item[iii)] If $\fA$ is a local Banach algebra and $n\ge 0$, then
there is a natural split monomorphism $K^{\top}_n(\fA)\to KH_n(I_{\ell^1(\fA)})$.
\end{thm}

\begin{proof}
Recall that $KH$ satisfies excision, vanishes on nilpotent rings and
commutes with filtering colimits (\cite{kh}). On the other hand,
$\ell^q(\fA)/\ell^p(\fA)$ is nilpotent for $p<q<\infty$ and

$$\ell^{r-}(\fA)=\colim_{s<r}\ell^s(\fA)\quad (0<r\le \infty).$$
It follows that $KH_*(I_{S(\fA)})=KH_*(I_{\ell^1(\fA)})$ for all $S$
as in the theorem. Recall also that $KH$ is $M_2$-stable. Hence
$KH_*(I_{\ell^1(-)})=KH_*(I_{\ell^p(-)})$ is H\"older-homotopy
invariant, by Theorem \ref{thm:htpy}. This proves i). By
\cite{kh}*{Proposition 1.5} (see also \cite{friendly}*{Proposition
5.2.3}), in order to prove ii) it suffices to show that
$I_{\ell^1(\fA)}$ is $K_0$-regular. By definition, a ring $A$ is
$K_0$-regular if for each $n\ge 1$ the canonical map
\[
K_0(A)\to K_0(A[t_1,\dots,t_n])
\]

is an isomorphism. This is equivalent to the requirement that for
$\ul{t}=(t_1,\dots,t_n)$, the map
\[
\epsilon:A[\ul{t}]\to A[\ul{t}], \ \ \epsilon(f)=f(0)
\]
induce an isomorphism in $K_0$. Observe that
\begin{align}\label{polii}
I_{\ell^1(\fA)}[\ul{t}]=&(\ell^1(\fA)\#_\cP\Gamma)[\ul{t}]\\
=&(\ell^1(\fA)[\ul{t}])\#_\cP\Gamma.\nonumber
\end{align}
Also note that, for the projective tensor product,
\begin{align}\label{nuclear}
\ell^1(C^\infty([0,1],\fA))=&\ell^1\hotimes
C^\infty([0,1],\C)\hotimes\fA\\
=&C^\infty([0,1],\ell^1(\fA)).\nonumber
\end{align}
Next we borrow an argument from \cite{roshan}*{Proposition 3.4}.
Consider the homomorphism
\begin{gather*}
\phi:C^\infty([0,1],\ell^1(\fA))[\ul{t}]\to
C^\infty([0,1],\ell^1(\fA))[\ul{t}]\\
\phi(f)(s,\ul{t})=f(s,s\ul{t}).
\end{gather*}
Using the identifications \eqref{polii} and \eqref{nuclear} we have
a diagram

\[
\xymatrix{I_{\ell^1(C^{\infty}([0,1],\fA))}[\ul{t}]\ar[r]^{\phi\#\Gamma}&
I_{\ell^1(C^{\infty}([0,1],\fA))}[\ul{t}]
\ar@/_1pc/[d]_{s=0}\ar@/^1pc/[d]^{s=1}\\
I_{\ell^1(\fA)}[\ul{t}]\ar[u]^\inc\ar@/^1pc/[r]^\epsilon\ar@/_1pc/[r]_{1}
&I_{\ell^1(\fA)}[\ul{t}] }
\]

One checks that both the outer and the inner square commute. By
Theorem \ref{thm:htpy},
$K_0(\ev_{s=0}\#\Gamma)=K_0(\ev_{s=1}\#\Gamma)$. It follows that
$K_0(\epsilon)$ is the identity; this proves ii). Next assume that
$\fA$ is a local Banach algebra; then $K_0^\top(\fA)=K_0(\fA)$. On the
other hand, by universal property of the crossed product, we have a
map
\begin{equation}\label{itoj}
I_{\ell^1(\fA)}=(\ell^1\hotimes\fA)\#_\cP\Gamma\to
\Ell^1\hotimes\fA.
\end{equation}
Composing this map with the inclusion
\begin{equation}\label{incluk}
\fA\to I_{\ell^1(\fA)},\ \ a\mapsto a E_{1,1},
\end{equation}
we obtain the map
\begin{equation}\label{map:incluj}
\fA\to \Ell^1\hotimes\fA,\quad a\mapsto a\hotimes E_{1,1}.
\end{equation}
Since the latter map induces an isomorphism in $K_0$, it follows
that \eqref{incluk} induces a split monomorphism $K_0(\fA)\to
K_0(I_{\ell^1(\fA)})$. Thus we have established iii) for $n=0$. For
the case $n\ge 1$, we consider the simplicial algebras of $C^\infty$
functions on the topological standard simplices and of polynomial
functions on the algebraic standard simplices:
$$\Delta^\dif:[n]\mapsto C^\infty(\Delta^n)$$ and
$$\Delta^\alg:[n]\mapsto\C[t_0,\dots,t_n]/\langle\sum
t_i-1\rangle.$$ Set
\begin{gather*}
\Delta^\dif\fA=\Delta^\dif\hotimes\fA\mbox{ and}\\
\Delta^\alg\fA=\Delta^\alg\otimes_\C\fA.
\end{gather*}
For $n\ge 1$, we have
\begin{gather*}
K^\top_n(\fA)=\pi_nBGL(\Delta^\dif \fA),\\
KV_n(\fA)=\pi_nBGL(\Delta^\alg \fA).
\end{gather*}
Hence for $KV(\fA)=BGL(\Delta^\alg\fA)$, there is a map
\begin{equation*}
K^\top_n(\fA)\to \pi_n(KV(\Delta^\dif(\fA))).
\end{equation*}
Composing the latter map with that induced by the inclusion
\eqref{incluk}, and using parts i) and ii), we get a homomorphism
\begin{equation}\label{ktotokhi}
K^{\top}_n(\fA)\to \pi_nKV(I_{\ell^1(\Delta^\dif\fA)})\cong
KV_n(I_{\ell^1(\fA)})=KH_n(I_{\ell^1(\fA)}).
\end{equation}
Composing \eqref{ktotokhi} with the homorphism induced by \eqref{itoj} we obtain
\begin{equation}\label{ktotokhj}
K_n^\top(\fA)\to KH_n(\Ell^1\hotimes\fA).
\end{equation}
But by \cite{cot}*{Theorem 6.2.1} the comparison map
\[
KH_n(\Ell^1\hotimes\fA)\to K^{\top}_n(\Ell^1\hotimes\fA)
\]
is an isomorphism. One checks that the latter map composed with
\eqref{ktotokhj} is equivalent to that induced by
\eqref{map:incluj}. But \eqref{map:incluj} induces an isomorphism in
$K^\top$ of local Banach algebras. This proves that \eqref{ktotokhi} is a
split monomorphism, concluding the proof.
\end{proof}

\begin{thm}\label{thm:kh0}
\item[i)] The functor $\bass\to \ab$, $\fA\mapsto KH_*(I_{c_0(\fA)})$
is invariant under continuous homotopies.

\item[ii)] For every bornological algebra $\fA$
\[
KH_n(I_{c_0(\fA)})=\left\{\begin{matrix}KV_n(I_{c_0(\fA)})& n\ge
1\\
K_n(I_{c_0(\fA)})& n\le 0.\end{matrix}\right.
\]

\item[iii)] If $\fA$ is a local $C^*$-algebra and $n\ge 0$, then
there is a natural split monomorphism $K^{\top}_n(\fA)\to KH_n(I_{c_0(\fA)})$.
\end{thm}
\begin{proof}
As in Theorem \ref{thm:kh}, part i) follows from Theorem
\eqref{thm:htpy}. To prove part ii), first observe that
\begin{align*}
c_0(C([0,1],\fA))=&C_0(\N,C([0,1],\fA))\\
=&C([0,1],c_0(\fA)).
\end{align*}
Then use the argument of the proof of part ii) of Theorem
\ref{thm:kh}. To prove part iii) first observe that if $\fA$ is a
local $C^*$-algebra, then for the spatial tensor product,
\[
c_0(\fA)=c_0\sotimes\fA.
\]
Hence if $\cK=\cK(\ell^2(\N))$ is the $C^*$-algebra of compact
operators then the map $\fA\to\fA\sotimes\cK$, $a\to a\otimes
E_{1,1}$ factors through $I_{c_0(\fA)}$. Taking this into account,
using the fact that, by \cite{sw1}*{Theorem 10.9} and
\cite{roshan}*{Proposition 3.4}, the comparison map
$KH_*(\fA\sotimes\cK)\to K^{\top}_*(\fA\sotimes\cK)$ is an
isomorphism, and substituting continuous functions for $C^\infty$
functions, we may now proceed as in the prooof of part iii) of
Theorem \ref{thm:kh}.
\end{proof}

\begin{rem}\label{rem:negnofunc}
The argument of the proofs of part iii) of Theorems \ref{thm:kh} and
\ref{thm:kh0} does not work for $n<0$. Indeed, $K_n$ and $K_n^\top$
do not agree for such $n$, not even on algebras on which the former
is homotopy invariant. For example negative $K$-theory is homotopy
invariant on commutative $C^*$-algebras (\cite{galgtop}*{Theorem
1.2}) yet $K_n(\C)=0$ for $n<0$, while $K_{2m}^{\top}(\C)=\Z$ for
$m\in\Z$.
\end{rem}

\begin{rem}\label{rem:khiso}
The argument of the proof of Theorem \ref{thm:kh} shows that if
$\fA$ is a local Banach algebra then $\fA\to\fA\hotimes\Ell^1$
factors through $I_{\ell^1(\fA)}$ and the map
\[
KH_n(I_{\ell^1{\fA}})\to KH_n(\fA\hotimes\Ell^1)=K_*^{\top}(\fA)
\]
is onto for $n\ge 0$. Similarly the argument of the proof of
\ref{thm:kh0} shows that for $\fA$ a local $C^*$-algebra  maps
$\fA\to\fA\sotimes\cK$ factors through $I_{c_0(\fA)}$ and
\[
KH_n(I_{c_0(\fA)})\to KH_n(\fA\sotimes\cK)=K^{\top}_*(\fA)
\]
is onto for $n\ge 0$.
\end{rem}

\subsection{\topdf{$K$}{K}-theory and cyclic homology}

\begin{thm}\label{thm:kseq}
Let $\fA$ be a bornological algebra and let $S$ be $c_0$, $\ell^p$,
$\ell^{p+}$ ($0<p<\infty$), or $\ell^{p-}$ ($0<p\le\infty$). Then
there are long exact sequences ($n\in\Z$)
\begin{equation}\label{seq:abs}
\xymatrix{
KH_{n+1}(I_{S(\fA)})\ar[r]&HC_{n-1}(I_{S(\fA)})\ar[d]\\
KH_{n}(I_{S(\fA)})&K_n(I_{S(\fA)})\ar[l] }
\end{equation}
and
\begin{equation}\label{seq:rel}
\xymatrix{
KH_{n+1}(I_{S(\fA)})\ar[r]&HC_{n-1}(\Gami(\fA):I_{S(\fA)})\ar[d]\\
KH_{n}(I_{S(\fA)})&K_n(\Gami(\fA):I_{S(\fA)})\ar[l] }
\end{equation}
\end{thm}
\begin{proof}
Let $K^{\nil}=\hofi(K\to KH)$ be the homotopy fiber of the
comparison map. By \cite{friendly}*{diagram (86)}, there is a
natural map $\nu:K^{\nil}(A)\to HC(A)[-1]$, defined for every
$\Q$-algebra $A$. Write $K^{\ninf}=\hofi(\nu)$; by
\cite{corel}*{Proposition 8.2.4} $K^{\ninf}$ is excisive,
$M_2$-stable and nilinvariant, and $K_*^{\ninf}$ commutes with
filtering colimits. Hence to prove the
theorem it suffices to show that
\begin{equation}\label{ninf0}
K^{\ninf}_*(I_{S(\fA)})=0.
\end{equation}
Note also that if $S\neq c_0$, then
$$K^{\ninf}_*(I_{S(\fA)})=K^{\ninf}_*(I_{\ell^1(\fA)})$$
by the same argument as that used in the proof of Theorem
\ref{thm:kh} to prove the analogue assertion for $KH$. Thus we may
assume from now on that $S\in\{c_0,\ell^1\}$. By
\cite{cot}*{Proposition 3.1.4}, to prove \eqref{ninf0} it suffices
to show that $I_{S(\fA)}$ is $K^{\inf}$-regular. Here $K^{\inf}$ is
infinitesimal $K$-theory; by \cite{kabi} it is excisive and
$M_2$-stable. Hence, the same argument as that used in the proof of
Theorems \ref{thm:kh} and \ref{thm:kh0} to prove that $I_{S(\fA)}$
is $K_0$-regular applies to show that it is also $K^{\inf}$-regular.
This completes the proof.
\end{proof}

\begin{rem}\label{rem:relsequot}
By Examples \ref{exas:infinitesums}, we have
\[
KH_*(\Gami(\fA))=HC_*(\Gami(\fA))=K_*(\Gami(\fA))=0
\]
for unital $\fA$. Hence in the unital case, the second sequence of Theorem \ref{thm:kseq} can be equivalently expressed
in terms of the quotient $\Gami(\fA)/I_{S(\fA)}$; we have a long exact sequence
\begin{equation}\label{seq:relquot}
\xymatrix{
KH_{n+1}(\Gami(\fA)/I_{S(\fA)})\ar[r]&HC_{n-1}(\Gami(\fA)/I_{S(\fA)})\ar[d]\\
KH_{n}(\Gami(\fA)/I_{S(\fA)})&K_n(\Gami(\fA)/I_{S(\fA)})\ar[l] }
\end{equation}
\end{rem}


\begin{bibdiv}
\begin{biblist}

\bib{busex}{article}{
   author={Buss, Alcides},
   author={Exel, Ruy},
   title={Fell bundles over inverse semigroups and twisted \'etale
   groupoids},
   journal={J. Operator Theory},
   volume={67},
   date={2012},
   number={1},
   pages={153--205},
}

\bib{calk}{article}{
   author={Calkin, J. W.},
   title={Two-sided ideals and congruences in the ring of bounded operators
   in Hilbert space},
   journal={Ann. of Math. (2)},
   volume={42},
   date={1941},
   pages={839--873},
}
\bib{cohn}{article}{
   author={Cohn, P. M.},
   title={Some remarks on the invariant basis property},
   journal={Topology},
   volume={5},
   date={1966},
   pages={215--228},
}
\bib{kabi}{article}{
   author={Corti{\~n}as, Guillermo},
   title={The obstruction to excision in $K$-theory and in cyclic homology},
   journal={Invent. Math.},
   volume={164},
   date={2006},
   number={1},
   pages={143--173},
}

\bib{friendly}{article}{
   author={Corti{\~n}as, Guillermo},
   title={Algebraic v. topological $K$-theory: a friendly match},
   conference={
      title={Topics in algebraic and topological $K$-theory},
   },
   book={
      series={Lecture Notes in Math.},
      volume={2008},
      publisher={Springer},
      place={Berlin},
   },
   date={2011},
   pages={103--165},
}

\bib{whatever}{article}{
   author={Corti{\~n}as, Guillermo},
   title={Cyclic homology, tight crossed products, and small stabilizations},
status={Preprint},
}

\bib{corel}{article}{
author={Corti{\~n}as, Guillermo},
author={Ellis, Eugenia},
title={Isomorphism conjectures with proper coefficients},
eprint={arXiv:1108.5196v3},
}

\bib{biva}{article}{
   author={Corti{\~n}as, Guillermo},
   author={Thom, Andreas},
   title={Bivariant algebraic $K$-theory},
   journal={J. Reine Angew. Math.},
   volume={610},
   date={2007},
   pages={71--123},
}

\bib{cot}{article}{
   author={Corti{\~n}as, Guillermo},
   author={Thom, Andreas},
   title={Comparison between algebraic and topological $K$-theory of locally
   convex algebras},
   journal={Adv. Math.},
   volume={218},
   date={2008},
   number={1},
   pages={266--307},
}


\bib{galgtop}{article}{
   author={Corti{\~n}as, Guillermo},
   author={Thom, Andreas},
   title={Algebraic geometry of topological spaces I.},
}
\bib{cmr}{book}{
   author={Cuntz, Joachim},
   author={Meyer, Ralf},
   author={Rosenberg, Jonathan M.},
   title={Topological and bivariant $K$-theory},
   series={Oberwolfach Seminars},
   volume={36},
   publisher={Birkh\"auser Verlag},
   place={Basel},
   date={2007},
   pages={xii+262},
}
\bib{cq}{article}{
   author={Cuntz, Joachim},
   author={Quillen, Daniel},
   title={Excision in bivariant periodic cyclic cohomology},
   journal={Invent. Math.},
   volume={127},
   date={1997},
   number={1},
   pages={67--98},
}


\bib{kenetal2}{article}{
   author={Dykema, Ken},
   author={Figiel, Tadeusz},
   author={Weiss, Gary},
   author={Wodzicki, Mariusz},
   title={Commutator structure of operator ideals},
   journal={Adv. Math.},
   volume={185},
   date={2004},
   number={1},
   pages={1--79},
}

\bib{ruy}{article}{
   author={Exel, Ruy},
   title={Inverse semigroups and combinatorial $C\sp \ast$-algebras},
   journal={Bull. Braz. Math. Soc. (N.S.)},
   volume={39},
   date={2008},
   number={2},
   pages={191--313},
}



\bib{garling}{article}{
   author={Garling, D. J. H.},
   title={On ideals of operators in Hilbert space},
   journal={Proc. London Math. Soc. (3)},
   volume={17},
   date={1967},
   pages={115--138},
}

\bib{hig}{article}{
   author={Higson, Nigel},
   title={Algebraic $K$-theory of stable $C\sp *$-algebras},
   journal={Adv. in Math.},
   volume={67},
   date={1988},
   number={1},
   pages={140},
}

\bib{kv}{article}{
   author={Karoubi, Max},
   author={Villamayor, Orlando},
   title={$K$-th\'eorie alg\'ebrique et $K$-th\'eorie topologique. I},
   journal={Math. Scand.},
   volume={28},
   date={1971},
   pages={265--307 (1972)},
}








\bib{roshan}{article}{
   author={Rosenberg, Jonathan},
   title={Comparison between algebraic and topological $K$-theory for Banach
   algebras and $C\sp *$-algebras},
   conference={
      title={Handbook of $K$-theory. Vol. 1, 2},
   },
   book={
      publisher={Springer},
      place={Berlin},
   },
   date={2005},
   pages={843--874},
}

\bib{simon}{book}{
   author={Simon, Barry },
   title={Trace Ideals and their Applications: Second Edition},
   series={Mathematical Surveys and Monographs},
   volume={120},
   publisher={AMS},
   date={2005},
}


\bib{sw1}{article}{
   author={Suslin, Andrei A.},
   author={Wodzicki, Mariusz},
   title={Excision in algebraic $K$-theory},
   journal={Ann. of Math. (2)},
   volume={136},
   date={1992},
   number={1},
   pages={51--122},
}

\bib{wag}{article}{
   author={Wagoner, J. B.},
   title={Delooping classifying spaces in algebraic $K$-theory},
   journal={Topology},
   volume={11},
   date={1972},
   pages={349--370},
}

\bib{kh}{article}{
   author={Weibel, Charles A.},
   title={Homotopy algebraic $K$-theory},
   conference={
      title={Algebraic $K$-theory and algebraic number theory (Honolulu, HI,
      1987)},
   },
   book={
      series={Contemp. Math.},
      volume={83},
      publisher={Amer. Math. Soc.},
      place={Providence, RI},
   },
   date={1989},
   pages={461--488},
}
\bib{wodk}{article}{
   author={Wodzicki, Mariusz},
   title={Algebraic $K$-theory and functional analysis},
   conference={
      title={First European Congress of Mathematics, Vol.\ II},
      address={Paris},
      date={1992},
   },
   book={
      series={Progr. Math.},
      volume={120},
      publisher={Birkh\"auser},
      place={Basel},
  },
   date={1994},
   pages={485--496},
}



\end{biblist}
\end{bibdiv}
\end{document}